\documentclass[12pt]{article}
\usepackage[T2A]{fontenc}
\usepackage[utf8]{inputenc}
\usepackage{amssymb,amsmath,amsfonts,amsthm,amscd,latexsym,indentfirst,verbatim}
\def\rb{\textrm{rb}}
\def\id{\textrm{id}}

\def\char{\textrm{char}\,}
\def\Var{\mathrm{Var}}
\def\Span{\mathrm{Span}}
\def\Aut{\mathrm{Aut}}
\def\Imm{\mathrm{Im}\,}
\def\Rad{\mathrm{Rad}\,}

\theoremstyle{plain}
\newtheorem{thm}{Theorem}[section]
\newtheorem{lem}[thm]{Lemma}
\newtheorem{cor}[thm]{Corollary}
\newtheorem{sta}[thm]{Statement}
\newtheorem{pro}[thm]{Proposition}
\newtheorem{con}[thm]{Conjecture}

\theoremstyle{definition}
\newtheorem{rem}{Remark}[section]
\newtheorem{exm}[rem]{Example}
\newtheorem{prob}[rem]{Problem}

\begin{document}

\hfill{16W99, 17C20 (MSC2010)}

\begin{center}
{\Large
Rota---Baxter operators on unital algebras}

V. Gubarev
\end{center}

\begin{abstract}
We state that all Rota---Baxter operators of nonzero weight
on the Grassmann algebra over a field of characteristic zero are
projections on a subalgebra along another one.
We show the one-to-one correspondence between the solutions
of associative Yang---Baxter equation and Rota---Baxter operators
of weight zero on the matrix algebra $M_n(F)$ (joint with P. Kolesnikov).

We prove that all Rota---Baxter operators of weight zero on
a unital associative (alternative, Jordan) algebraic algebra
over a field of characteristic zero are nilpotent.
We introduce a new invariant for an algebra~$A$ called the RB-index
$\rb(A)$ as the minimal nilpotency index of Rota---Baxter operators
of weight zero on~$A$. We show that $\rb(M_n(F)) = 2n-1$ provided that characteristic of $F$ is zero.

\medskip
{\it Keywords}:
Rota---Baxter operator, Yang---Baxter equation, matrix algebra,
Grassmann algebra, Faulhaber polynomial.
\end{abstract}

\tableofcontents

\section{Introduction}

Given an algebra $A$ and a scalar $\lambda\in F$, where $F$ is a~ground field,
a~linear operator $R\colon A\rightarrow A$ is called a Rota---Baxter operator
(RB-operator, for short) of weight~$\lambda$ if the identity
\begin{equation}\label{RB}
R(x)R(y) = R( R(x)y + xR(y) + \lambda xy )
\end{equation}
holds for all $x,y\in A$.
An algebra~$A$ equipped with a~Rota---Baxter operator is called
a~Rota---Baxter algebra (RB-algebra).

The notion of a Rota---Baxter operator was introduced
by G.~Baxter~\cite{Baxter} in 1960 as formal generalization
of integration by parts formula and then developed
by G.-C.~Rota~\cite{Rota} and others \cite{Atkinson,Cartier}.

In 1980s, the deep connection between constant solutions
of the classical Yang---Baxter equation from mathematical
physics and RB-operators of weight zero on a semi\-simple
finite-dimensional Lie algebra was discovered \cite{BelaDrin82}.
In 2000, M. Aguiar stated \cite{Aguiar00} that solutions
of associative Yang---Baxter equation (AYBE, \cite{Zhelyabin})
and RB-operators of weight zero on any associative algebra are related.

In 2000, the connection between RB-algebras and prealgebras
was found \cite{Aguiar00,GolubchikSokolov}.
Later, this connection was extended and studied for postalgebras,
see, e.g., \cite{BBGN2011,Embedding,GuoMonograph}.

To the moment, applications of Rota---Baxter operators
in symmetric polynomials, quantum field renormalization,
shuffle algebra, etc. were found \cite{Atkinson,Renorm,GuoMonograph,Shuffle,Ogievetsky}.

Also, RB-operators have been studied by their own interest.
RB-opera\-tors were classified on $\mathrm{sl}_2(\mathbb{C})$
\cite{Kolesnikov,KonovDissert,sl2,sl2-0}, $M_2(\mathbb{C})$ \cite{BGP,Mat2},
$\mathrm{sl}_3(\mathbb{C})$ \cite{KonovDissert},
the Grassmann algebra $\mathrm{Gr}_2$ \cite{BGP,HuaExterior},
the 3-dimensional simple Jordan algebra of bilinear form,
the Kaplansky superalgebra $\mathrm{K}_3$~\cite{BGP},
3-dimensional solvable Lie algebras \cite{KonovDissert,LieSolv1,LieSolv2},
low-dimensional Lie superalgebras~\cite{Heisenberg18,Filiform,Heisenberg17,Heisenberg16},
low-dimensional pre-Lie (super)algebras~\cite{preLiesuper,preLie},
low-dimensional semigroup algebras~\cite{Semigroup}.
The classification of RB-operators of special kind
on polynomials, power series and Witt algebra was
found~\cite{Witt,PowerSeries,Monom2,SkewPower,WittNonHom,Monom}.

Author proceeds on the study of Rota---Baxter operators initiated in \cite{BGP,Gub2017}.
Let us give a brief outline of the work.
In \S2 the required preliminaries are stated.
In \S3 it is proved that Aguiar's correspondence
between the solutions of AYBE and RB-operators
of weight zero on the matrix algebra $M_n(F)$ is bijective
(Theorem~\ref{thm4}, joint result with P. Kolesnikov).

Further, in \S4, we consider RB-operators of nonzero weight.
There is a big difference between RB-operators of nonzero weight
and RB-operators of weight zero.
There are few known general constructions for the first ones
such as splitting and triangular-splitting RB-operators.
In contrast, there are a lot of examples for the second ones,
and it is not clear which of them are of most interest.
Thus, we start exposition and study of RB-operators with the nonzero case.
In \cite{BGP}, it was proved that any RB-operator
of nonzero weight on an odd-dimensional simple Jordan algebra~$J$
of bilinear form is splitting, i.e., it is a projection on a subalgebra
along another one provided that $J$ splits into a direct
vector-space sum of these two subalgebras.

Given an algebra~$A$, an RB-operator $R$ on~$A$ of weight~$\lambda$,
and $\psi\in\Aut(A)$, the operator $R^{(\psi)} = \psi^{-1}R\psi$
is an RB-operator of weight~$\lambda$ on~$A$.
Thus, it is reasonable to classify RB-operators
of any weight on an algebra only up to conjugation
with automorphisms.
Since the map $\phi$ defined as $\phi(R) = -(R+\lambda\id)$
preserves the set of all RB-operators of the given weight~$\lambda$,
we study RB-operators of nonzero weight up to the action of~$\phi$.

The combinatorics arisen from the substitution of unit into the RB-identity implies

{\bf Theorem A} (= Theorem~\ref{thm9}).
Let $A$~be a~unital power-associative algebra over a~field
of characteristic zero and $A = F1\oplus N$ (as vector spaces),
where $N$ is a~nil-algebra. Then each RB-operator $R$ of nonzero weight on~$A$
is splitting and up to $\phi$ we have $R(1) = 0$.

By Theorem~A, we get that all RB-operators of nonzero weight on
the Grassmann algebra are splitting (Corollary~\ref{cor4}, \S4.4).

The second main result in the case of nonzero weight gives us
a~powerful tool to study RB-operators of nonzero weight on the matrix algebra.

{\bf Theorem B} (= Theorem~\ref{thm11}).
Given an algebraically closed field $F$ of characteristic zero
and an RB-operator $R$ of nonzero weight~$\lambda$ on $M_n(F)$,
there exists a $\psi\in\Aut(M_n(F))$ such that
the matrix $R^{(\psi)}(1)$ is diagonal and up to $\phi$
the set of its diagonal elements has a form
$\{-p\lambda,(-p+1)\lambda,\ldots,-\lambda,0,\lambda,\ldots,q\lambda\}$
with $p,q\in\mathbb{N}$.

As a corollary, we show that an RB-operator on $M_3(F)$
up to conjugation with an automorphism
preserves the subalgebra of diagonal matrices (Corollary~\ref{Cor:diagM3}, \S4.5).

In \S5, we study RB-operators of weight zero.
In \cite{preLie}, the following question is raised:
``Whether we can give a meaningful ``classification rules''
so that the classification of Rota-Baxter operators can be more ``interesting''?

In Lemma~\ref{lem9} (\S5.1), we present, maybe, the first systematic attempt
to study possible constructions of RB-operators of weight zero.
For example,

{\bf Lemma} (= Lemma~\ref{lem9}(a)).
Given an algebra $A$, a linear operator~$R$
on~$A$ is an RB-operator of weight zero, if $A = B\oplus C$ (as vector spaces)
with $\Imm R$ having trivial multiplication,
$R\colon B\to C$, $R\colon C\to (0)$, $BR(B),R(B)B\subseteq C$.

This construction works for any variety and any dimension of algebras.
Let us list some examples of RB-operators of weight zero arisen from Lemma directly:

\begin{itemize}
\item \cite{Aguiar00-2,Jian}
Let $A$ be an associative algebra and $e\in Z(A)$ be an element such that $e^2 = 0$.
A~linear map $l_e\colon x\to ex$ is an RB-operator of weight~0.

\item
In \cite{Embedding}, given a~variety $\Var$ of algebras and a pre-$\Var$-algebra~$A$,
an enveloping RB-algebra~$B$ of weight zero and the variety $\Var$ for~$A$ was constructed.
By the construction, $B = A\oplus A'$ (as vector spaces),
where $A'$ is a copy of $A$ as a vector space,
and the RB-operator $R$ was defined as follows: $R(a')= a$, $R(a)=0$, $a\in A$.

\item In \cite{Witt}, all homogeneous RB-operators on the Witt algebra
$W = \Span\{L_n \mid n\in\mathbb{Z}\}$ over $\mathbb{C}$ with the Lie product
$[L_m,L_n] = (m-n)L_{m+n}$ were described.
A~homogeneous RB-operator with degree $k\in\mathbb{Z}$ satisfies the condition
$R(L_m)\in\Span\{L_{m+k}\}$ for all $m\in\mathbb{Z}$.
One of the four homogeneous RB-operators on~$W$,

(W2) $R(L_m) = \delta_{m+2k,0}L_{m+k}$, $k\neq0$, $m\in\mathbb{Z}$,

is defined by Lemma.

\item \cite{BGP} All RB-operators of weight zero on the Grassmann algebra $\mathrm{Gr}_2$
over a field $F$ with $\char F\neq2$ are defined by Lemma.

\item \cite{Aguiar00-2,BGP,Mat2}
There are exactly four nonzero RB-operators of weight zero
on $M_2(F)$ over an algebraically closed field~$F$ up to conjugation
with automorphisms $M_2(F)$, transpose and multiplication
on a nonzero scalar. One of them,

(M1) $R(e_{21}) = e_{12}$, $R(e_{11}) = R(e_{12}) = R(e_{22}) = 0$

is defined by Lemma.
\end{itemize}

As we know, we cover in Lemma~\ref{lem9} most of currently known examples of RB-operators of weight zero.

In \S5.2, we find the general property of RB-operators of weight zero on unital algebras.

{\bf Theorem C} (= Theorem~\ref{thm12}).
Let $A$ be a unital associative (alternative, Jordan)
algebraic algebra over a~field of characteristic zero.
Then there exists $N$ such that $R^N = 0$
for every RB-operator~$R$ of weight zero on $A$.

Let us compare the obtained results in both zero and nonzero cases.
Theorems A and B say that the nature of an RB-operator of nonzero weight
on a unital algebra is projective-like, it acts identically on a some subalgebra.
To the contrary, Theorem C says that an RB-operator of weight zero on a unital algebra is nilpotent.

After Theorem~C, we define a new invariant of an algebra~$A$, namely,
Rota---Baxter index (RB-index) $\rb(A)$ as the minimal
natural number with such nilpotency property.
We prove that $\rb(M_n(F)) = 2n-1$ over a field~$F$
of characteristic zero (\S5.4) and study RB-index for the Grassmann algebra (\S5.3),
unital composition algebras and simple Jordan algebras (\S5.5).

Let us clarify the title of the paper.
All three main results (Theorems A, B, and C) come from the combinatorics
arisen from multiple substitutions of the unit into the RB-identity.
So, these results are not valid when an algebra is not unital.
Also, we study typical constructions of RB-operators
on unital algebras.
For completeness, we add few examples of Lie algebras such that
RB-operators on them are arisen from the same constructions.

The proofs essentially involve the results on maximal subspaces
and subalgebras of special kind in the matrix
algebra~\cite{Agore,Meshulam,Wiegmann,Quater}
as well as the structure theory of associative and nonassociative algebras.

\section{Preliminaries}

In~\S2, we give some known examples and state useful properties of Rota---Baxter operators.

Let $F$ stand for the ground field.
By algebra we mean a vector space endowed with a~bilinear
(but not necessary associative) product.
All algebras are considered over~$F$.

Trivial RB-operators of weight $\lambda$ are zero operator and $-\lambda\id$.

\begin{sta}[\cite{GuoMonograph}]\label{sta1}
Given an RB-operator $P$ of weight $\lambda$,

(a) the operator $-P-\lambda\id$ is an RB-operator of weight $\lambda$,

(b) the operator $\lambda^{-1}P$ is an RB-operator of weight 1, provided $\lambda\neq0$.
\end{sta}

Given an algebra $A$, let us define a map $\phi$ on the set of all RB-operators on $A$
as $\phi(P)=-P-\lambda(P)\id$, where $\lambda(P)$ denotes the weight of an RB-operator~$P$.
It is clear that $\phi^2$ coincides with the identity map.

\begin{sta}[\cite{BGP}]\label{sta2}
Given an algebra $A$, an RB-operator $P$ on $A$ of weight $\lambda$,
and $\psi\in\Aut(A)$, the operator $P^{(\psi)} = \psi^{-1}P\psi$
is an RB-operator of weight~$\lambda$ on~$A$.
\end{sta}

The same result is true when $\psi$ is an antiautomorphism of $A$, i.e.,
a bijection from $A$ to $A$ satisfying $\psi(xy) = \psi(y)\psi(x)$ for all $x,y\in A$;
e.g., transpose on the matrix algebra.

\begin{sta}[\cite{GuoMonograph}]\label{sta3}
Let an algebra $A$ split as a vector space
into a direct sum of two subalgebras $A_1$ and $A_2$.
An operator $P$ defined as
\begin{equation}\label{Split}
P(a_1 + a_2) = -\lambda a_2,\quad a_1\in A_1,\ a_2\in A_2,
\end{equation}
is RB-operator of weight~$\lambda$ on~$A$.
\end{sta}

Let us call an RB-operator from Statement~\ref{sta3} as
{\it splitting} RB-operator with subalgebras $A_1,A_2$.
Note that the set of all splitting RB-operators on
an algebra $A$ is in bijection with all decompositions of~$A$
into a~direct sum of two subalgebras $A_1,A_2$.

\begin{rem}
Given an algebra $A$, let $P$ be a splitting RB-operator of weight~$\lambda$
on~$A$ with subalgebras $A_1,A_2$.
Hence, $\phi(P)$ is an RB-operator of weight $\lambda$ and
$$
\phi(P)(a_1+a_2) = -\lambda a_1,\quad a_1\in A_1,\ a_2\in A_2.
$$
So $\phi(P)$ is splitting RB-operator with the same subalgebras $A_1, A_2$.
\end{rem}

\begin{exm}\label{exm1}
(a) \cite{Aguiar00-2,Jian}
Let $A$ be an associative algebra and $e\in A$ an element such that
$e^2 = - \lambda e$, $\lambda\in F$.
A~linear map $l_e\colon x\to ex$ is an RB-operator of weight~$\lambda$
satisfying $l_e^2 + \lambda l_e = 0$.
For $\lambda\neq0$, $l_e$ is a splitting RB-operator on~$A$
with the subalgebras $A_1 = (1-e)A$ and $A_2 = eA$,
the decomposition $A = A_1\oplus A_2$ (as vector spaces) is the~Pierce one.

(b) \cite{BGP} In an alternative algebra~$A$ with an element~$e$ such that
$e^2 = - \lambda e$, $\lambda\in F$, the operator $l_e$ is an RB-operator on~$A$
if $e$~lies in associative or commutative center of~$A$.
\end{exm}

\begin{exm}
In \cite{BBGN2011}, it was proved that every RB-algebra of nonzero weight~$\lambda$
and a~variety $\Var$ under the operations
$x\succ y = R(x)y$,
$x\prec y = xR(y)$,
$x\cdot y = \lambda xy$
is a~post-$\Var$-algebra.
In \cite{Embedding}, given a post-$\Var$-algebra $A$,
its enveloping RB-algebra $B$ of weight~$\lambda$ and the variety
$\Var$ was constructed. By the construction, $B = A\oplus A'$ (as vector spaces),
where $A'$ is a copy of $A$ as of a~vector space,
and the RB-operator $R$ was defined as follows:
$R(a')=\lambda a$, $R(a)=-\lambda a$, $a\in A$.
Note that $R$ is the splitting RB-operator on $B$ with
$A_1 =\Span\{a+a'\mid a\in A\}$ and $A_2 = A$.
\end{exm}

\begin{exm}
Due to \cite{Monom}, all nontrivial monomial RB-operators of nonzero weight on $F[x]$
(i.e., mapping each monomial to a monomial) up to the action of $\phi$ and $\mathrm{Aut}(F[x])$
are splitting with subalgebras $F$ and $\mathrm{Id}\langle x\rangle$.
\end{exm}

\begin{exm}
By \cite{AnBai}, all RB-operators of nonzero weight on the associative algebra
$A = \Span\{e_1,\ldots,e_n\mid e_1e_i = e_i,e_je_i = 0,i=1,\ldots,n,j=2,\ldots,n\}$
are splitting.
\end{exm}

Given an algebra $A$ with the product $\cdot$, define the operations $\circ$, $[,]$ on
the space $A$:
$$
a\circ b=a\cdot b+b\cdot a,\quad [a,b]=a\cdot b-b\cdot a.
$$
We denote the space $A$ endowed with $\circ$ as $A^{(+)}$
and the space $A$ endowed with $[,]$ as~$A^{(-)}$.

\begin{sta}[\cite{BGP}]\label{sta5}
Let $A$ be an algebra.

(a) If $R$ is an RB-operator of weight~$\lambda$ on~$A$,
then $R$ is an RB-operator of weight~$\lambda$ on $A^{(\pm)}$.

(b) If all RB-operators of nonzero weight on $A^{(+)}$
(or $A^{(-)}$) are splitting, then all
RB-operators of nonzero weight on $A$ are splitting.
\end{sta}

\begin{lem}[\cite{BGP,GuoMonograph}]\label{lem1}
Let $A$ be a unital algebra, $P$ be an RB-operator of weight~$\lambda$ on~$A$.

(a) If $\lambda\neq0$ and $P(1)\in F$, then $P$ is splitting with one of subalgebras being unital;

(b) if $\lambda = 0$, then $1\not\in \Imm P$.
Moreover, if $A$ is simple finite-dimensional algebra, $\dim A>1$,
then $\dim (\ker P)\geq2$;

(c) if $\lambda = 0$ and $P(1)\in F$, then $P(1) = 0$,
$P^2 = 0$, and $\Imm P\subset \ker P$;

(d) if $\lambda = 0$ and $A$ is a power-associative algebra,
then $(P(1))^n = n! P^n(1)$, $n\in\mathbb{N}$.
\end{lem}

\begin{lem}\label{lem2}
Let an algebra $A$ be equal to a~direct sum of two ideals $A_1$, $A_2$,
and let $R$~be an RB-operator of weight~$\lambda$ on~$A$.
Then $\mathrm{Pr}_i R$ is the RB-operator of weight~$\lambda$ on $A_i$,
$i=1,2$. Here $\mathrm{Pr}_i$ denotes the projection from $A$ onto $A_i$.
\end{lem}

\begin{proof}
Straightforward.
\end{proof}

\begin{thm}[\cite{BGP}]\label{thm1}
All RB-operators on a quadratic division algebra are trivial.
\end{thm}

For RB-operators of weight zero, the last result may be generalized as follows.

\begin{thm}\label{thm2}
If $R$ is an RB-operator of weight zero on an algebraic power-associative algebra $A$ without zero divisors,
then $R = 0$.
\end{thm}

\begin{proof}
Suppose $R(x)\neq0$ for some $x\in A$.
As $A$ is algebraic, consider the equality
$$
(R(x))^m + \alpha_{m-1}(R(x))^{m-1}+\ldots +\alpha_1 R(x) + \alpha_0 = 0, \quad \alpha_i\in F,
$$
of minimal degree $m\geq1$. By Lemma~\ref{lem1}(b), $\alpha_0 = 0$.
Thus, $R(x)y = 0$ for $y = (R(x))^{m-1} + \alpha_{m-1}(R(x))^{m-2}+\ldots +\alpha_1$.
As $R(x)\neq0$, we have $y = 0$ and $m$ is not minimal.
\end{proof}

\section{Yang---Baxter equation}

In~\S3, we consider connections of Rota---Baxter operators with
different versions of Yang---Baxter equation.
In Theorem~\ref{thm4}, we state the bijection between the solutions
of associative Yang---Baxter equation and RB-operators
of weight zero on the matrix algebra $M_n(F)$ (joint with P. Kolesnikov).

\subsection{Classical Yang---Baxter equation}

Let $L$ be a semisimple finite-dimensional Lie algebra over $\mathbb{C}$.
For $r = \sum a_i\otimes b_i\in L\otimes L$, introduce
classical Yang---Baxter equation (CYBE, \cite{BelaDrin82}) as
\begin{equation}\label{CYBE}
[r_{12},r_{13}]+[r_{12},r_{23}]+[r_{13},r_{23}] = 0,
\end{equation}
where
$$
r_{12} = \sum a_i\otimes b_i\otimes 1,\quad
r_{13} = \sum a_i\otimes 1\otimes b_i,\quad
r_{23} = \sum 1\otimes a_i\otimes b_i
$$
are elements from {$U(L)^{\otimes 3}$}.

The switch map $\tau\colon L\otimes L\to L\otimes L$
acts in the following way: $\tau(a\otimes b) = -b\otimes a$.
The solution $r$ of CYBE is called skew-symmetric if $r + \tau(r) = 0$.
A~linear map $R\colon L\to L$ defined by a~skew-symmetric solution~$r$ of CYBE as
\begin{equation}\label{CYBE2RB}
R(x) = \sum \langle a_i,x\rangle b_i
\end{equation}
is an RB-operator of weight zero on $L$ \cite{BelaDrin82}.
Here $\langle \cdot,\cdot\rangle$ denotes the Killing form on $L$.

\begin{exm}\label{exm5}
There exists unique (up to conjugation and scalar multiple)
nonzero skew-symmetric solution of CYBE on $\mathrm{sl}_2(\mathbb{C})$:
$e\otimes h - h\otimes e$ \cite{Stolin}. It corresponds to the RB-operator
$R(e) = 0$, $R(f) = 4h$, $R(h) = -8e$.
\end{exm}

Given a finite-dimensional semisimple Lie algebra~$L$,
let us call a~linear map~$P$ a~skew-symmetric one if so is
its matrix in the orthonormal basis
(with respect to the Killing form).

\begin{sta}[\cite{GuoMonograph,sl2-0}]\label{staNew}
Given a finite-dimensional semisimple Lie algebra~$L$,
skew-symmetric solutions of CYBE on~$L$ are in one-to-one
correspondence with skew-symmetric Rota---Baxter operators
of weight zero on~$L$.
\end{sta}

An element $r\in L^{\otimes n}$, $n\in\mathbb{N}$, is called
$L$-invariant if $[r,y] = 0$ for all $y \in L$.
Here $L$ acts on $L^{\otimes n}$ by the formula
$[x_1 \otimes \ldots \otimes x_n, y]
= \sum\limits_{i=1}^n x_1 \otimes \ldots \otimes [x_i,y] \otimes \ldots \otimes x_n$, $x_i,y \in L$.

\begin{thm}[\cite{Goncharov2}]\label{thm3}
Let $L$ be a simple finite-dimensional Lie algebra over $\mathbb{C}$
and $r = \sum\limits_i a_i\otimes b_i$
be a non skew-symmetric solution of CYBE such that
$r+\tau(r)$ is $L$-invariant. Then the linear map $R\colon L\rightarrow L$
defined by~(\ref{CYBE2RB}) is an RB-operator of nonzero weight.
\end{thm}

\begin{exm}[\cite{Goncharov2}]\label{exm6}
Let $L = \mathrm{sl}_2(\mathbb{C})$ with the Chevalley basis $e,f,h$.
Consider an element
$$
r = \alpha(h\otimes e - e\otimes h) +\frac{1}{4}h\otimes h
  + e\otimes f\in L\otimes L,\quad \alpha\in \mathbb{C}.
$$
For any $\alpha\in\mathbb{C}$, the tensor $r$ is the
non skew-symmetric solution of CYBE and $r+\tau(r)$ is $L$-invariant.
Due to Theorem~\ref{thm3}, we get the RB-operator~$R$ on $\mathrm{sl}_2(\mathbb{C})$
defined by~$r$ as follows:
$R(e) = 0$, $R(h) = 2h+8\alpha e$, $R(f) = 4(f-\alpha h)$.
The weight of $R$ equals $-4$.
\end{exm}

\subsection{Modified Yang---Baxter equation}

In 1983 \cite{Semenov83}, Semenov-Tyan-Shansky introduced
modified Yang---Baxter equation
(MYBE)\footnote{There is at least one another version of Yang---Baxter equation
defined by M.~Gerstenhaber~\cite{Gerstenhaber} which is also called modified Yang---Baxter equation.}
as follows:
let $L$ be a Lie algebra, $R$ a linear map on $L$, then
\begin{equation}\label{MYBE}
R(x)R(y) - R(R(x)y+xR(y)) = -xy.
\end{equation}

It is easy to check that $R$ is a solution of MYBE if and only if
$R+\id$ is an RB-operator of weight $-2$.
So, there is one-to-one correspondence
(up to scalar multiple and action of $\phi$)
between the set of solutions of MYBE and
RB-operators of nonzero weight.

In \cite{Semenov83}, the general approach for solving MYBE
on a simple finite-dimen\-sional Lie algebra $L$ over
$\mathbb{C}$ was developed.
Applying this method, all solutions of MYBE
(i.e., all RB-operators of nonzero weight) on
$\mathrm{sl}_2(\mathbb{C})$ and $\mathrm{sl}_3(\mathbb{C})$
were found in \cite{KonovDissert}.

\subsection{Associative Yang---Baxter equation}

Let $A$ be an associative algebra, $r = \sum a_i\otimes b_i\in A\otimes A$.
The tensor $r$ is a solution of associative Yang---Baxter equation
(AYBE, \cite{Aguiar00-2,Aguiar01,Polishchuk,Zhelyabin}) if
\begin{equation}\label{AYBE}
r_{13}r_{12}-r_{12}r_{23}+r_{23}r_{13} = 0,
\end{equation}
where the definition of $r_{12},r_{13},r_{23}$ is the same as for CYBE.

A solution $r$ of AYBE on an algebra $A$ is a solution of CYBE on $A^{(-)}$
provided that $r + \tau(r)$ is $A$-invariant \cite{Aguiar01}.
A~tensor $u\otimes v\in A\otimes A$ is said to be $A$-invariant if
$au\otimes v = u\otimes va$ for all $a\in A$. In particular,
each skew-symmetric solution of AYBE is a~skew-symmetric solution of CYBE on $A^{(-)}$.

\begin{sta}[\cite{Aguiar00}]\label{sta4}
Let $r = \sum a_i\otimes b_i$ be a solution of AYBE
on an associative algebra~$A$.
A linear map $P_r\colon A\to A$ defined as
\begin{equation}\label{AYBE2RB}
P_r(x) = \sum a_i x b_i
\end{equation}
is an RB-operator of weight zero on $A$.
\end{sta}

\begin{exm}[\cite{Aguiar00-2}]\label{exm7}
Up to conjugation, transpose and scalar multiple,
all non\-zero solutions of AYBE on $M_2(\mathbb{C})$ are
$(e_{11}+e_{22})\otimes e_{12}$;\
$e_{12}\otimes e_{12}$;\
$e_{22}\otimes e_{12}$;\
$e_{11}\otimes e_{12} - e_{12}\otimes e_{11}$.
\end{exm}

All RB-operators arisen by Statement~\ref{sta4}
from the solutions of AYBE on $M_2(\mathbb{C})$
are exactly all RB-operators of weight zero on $M_2(\mathbb{C})$ (see Theorem~\ref{thm15}, \S5.4).
Let us generalize this fact on each matrix algebra $M_n(F)$,
this result is a joint one with P.~Kolesnikov.

\begin{thm}\label{thm4}
The map $r\to P_r$ is the bijection between the set of the
solutions of AYBE on $M_n(F)$ and the set of RB-operators
of weight zero on $M_n(F)$.
\end{thm}

\begin{proof}
Given a linear operator $R$ on $M_n(F)$, denote
$R(e_{pq}) = \sum\limits_{i,l}t_{ip}^{ql}e_{il}$.
Then the equation~(\ref{RB}) for $R$ could be rewritten in the form
\begin{equation}\label{RB-tensor}
\sum_{j}\big(t_{ia}^{bj}t_{jc}^{dl}
           - t_{ja}^{bc}t_{ij}^{dl}
           - t_{bc}^{dj}t_{ia}^{jl} \big) = 0.
\end{equation}

Let $r = \sum\limits_{i,j,k,l}s_{ij}^{kl}e_{ij}\otimes e_{kl}$ be a solution of AYBE. So,
\begin{equation}\label{AYBE-tensor-prep}
\begin{gathered}
r_{13}r_{12} = \sum\limits_{p}s_{ij}^{kl}s_{pi}^{st}e_{pj}\otimes e_{kl}\otimes e_{st}, \\
r_{12}r_{23} = \sum\limits_{p}s_{ij}^{kl}s_{lp}^{st}e_{ij}\otimes e_{kp}\otimes e_{st}, \\
r_{23}r_{13} = \sum\limits_{p}s_{ij}^{kl}s_{pr}^{lt}e_{pr}\otimes e_{ij}\otimes e_{kt}.
\end{gathered}
\end{equation}
By substituting the summands from~(\ref{AYBE-tensor-prep}) into~(\ref{AYBE})
and gathering them on the tensor $e_{ij}\otimes e_{kl}\otimes e_{st}$,
we get the equality
\begin{equation}\label{AYBE-tensor}
\sum_{p}\big(s_{pj}^{kl}s_{ip}^{st}
           - s_{ij}^{kp}s_{pl}^{st}
           + s_{kl}^{sp}s_{ij}^{pt} \big) = 0.
\end{equation}
By the interchange of variables, the equations~(\ref{RB-tensor}) and~(\ref{AYBE-tensor}) coincide.
Thus, the map $\chi$ from the set of the solutions of AYBE on $M_n(F)$ to the set
of RB-operators of weight zero on $M_n(F)$ acting as $\chi(r) = T_r$,
where $r = \sum\limits_{i,j,k,l}s_{ij}^{kl}e_{ij}\otimes e_{kl}$
and $T_r(e_{pq}) = \sum\limits_{i,l}s_{ip}^{ql}e_{il}$, is the bijection.
It remains to note that $\chi$ is exactly the map $r\to P_r$.
\end{proof}

\begin{thm}[\cite{Sokolov}]
Up to conjugation, transpose and scalar multiple all
nonzero skew-symmetric solutions of AYBE on $M_3(\mathbb{C})$ are

(A1) $e_{32}\otimes e_{31} - e_{31}\otimes e_{32}$,

(A2) $e_{11}\otimes e_{12} - e_{12}\otimes e_{11}
 + e_{13}\otimes (e_{12}-e_{21}) - (e_{12}-e_{21})\otimes e_{13}
 + (e_{11}+e_{22})\otimes e_{23} - e_{23}\otimes (e_{11}+e_{22})$,

(A3) $e_{22}\otimes e_{23} - e_{23}\otimes e_{22}$,

(A4) $e_{13}\otimes (e_{12}-e_{21}) - (e_{12} - e_{21})\otimes e_{13}
 + (e_{11} + e_{22})\otimes e_{23} - e_{23}\otimes (e_{11} + e_{22})$,

(A5) $(e_{11}+e_{22})\otimes e_{23} - e_{23}\otimes (e_{11}+e_{22})
  + e_{21}\otimes e_{13} - e_{13}\otimes e_{21}$,

(A6) $(e_{11} + e_{33})\otimes e_{23} - e_{23}\otimes (e_{11} + e_{33})
 + e_{11}\otimes e_{13} - e_{13}\otimes e_{11}$,

(A7) $e_{13}\otimes e_{21} - e_{21}\otimes e_{13} + e_{33}\otimes e_{23} - e_{23}\otimes e_{33}$,

(A8) $(e_{11}+e_{33})\otimes e_{23} - e_{23}\otimes(e_{11}+e_{33})$.
\end{thm}

We call an RB-operator $R$ on $M_n(F)$ a skew-symmetric one if $R^* = -R$,
where $R^*$ is the conjugate operator relative to the trace form.

\begin{cor}\label{cor1}
Up to conjugation, transpose and scalar multiple all
non\-zero skew-sym\-metric RB-operators on $M_3(\mathbb{C})$ are

(R1) $R(e_{31}) = e_{23}$, $R(e_{32}) = - e_{13}$;

(R2) $R(e_{11}) = - e_{21} - e_{32}$, $R(e_{12}) = e_{11} + e_{31}$,
$R(e_{13}) = e_{12} - e_{21}$, $R(e_{21}) = - e_{31}$, $R(e_{22}) = - e_{32}$,
$R(e_{23}) = e_{11} + e_{22}$;

(R3) $R(e_{23}) = e_{22}$, $R(e_{22}) = - e_{32}$;

(R4) $R(e_{13}) = e_{12} - e_{21}$, $R(e_{12}) = - R(e_{21}) = e_{31}$,
$R(e_{23}) = e_{11} + e_{22}$, $R(e_{11}) = R(e_{22}) = - e_{32}$;

(R5) $R(e_{13}) = e_{12}$, $R(e_{21}) = - e_{31}$, $R(e_{23}) = e_{11} + e_{22}$, $R(e_{11}) = R(e_{22}) = - e_{32}$;

(R6) $R(e_{33}) = e_{32}$, $R(e_{23}) = - e_{33}$, $R(e_{13}) = e_{11} + e_{12}$, $R(e_{11}) = R(e_{21}) = - e_{31}$;

(R7) $R(e_{23}) = -e_{11}-e_{33}$, $R(e_{11}) = R(e_{33}) = e_{32}$;

(R8) $R(e_{13}) = e_{12}$, $R(e_{21}) = -e_{31}$, $R(e_{33}) = e_{32}$, $R(e_{23}) = - e_{33}$.
\end{cor}

In 2006, K.~Ebrahimi-Fard defined in his Thesis~\cite[p.~113]{FardThesis}
the associative Yang---Baxter equation of weight~$\lambda$.
Later, this equation was twice rediscovered: in 2010 in~\cite{Ogievetsky} and in 2018 in~\cite{AYBE-ext}.
Given an associative algebra~$A$ and a tensor $r\in A\otimes A$,
we say that $r$~is a~solution of associative Yang---Baxter equation
of weight~$\lambda$ (wAYBE) if
\begin{equation}\label{wAYBE}
r_{13}r_{12}-r_{12}r_{23}+r_{23}r_{13} = \lambda r_{13}.
\end{equation}

In~\cite{FardThesis,AYBE-ext},
it was shown that the solutions of wAYBE generate $\epsilon$-unitary bialgebras.
Moreover, the following analogues of Statement~\ref{sta4} and Theorem~\ref{thm4} hold,
the last one was stated as the generalization of~Theorem~\ref{thm4}.

\begin{sta}[\cite{Anghel,FardThesis,AYBE-ext}]
Let $r = \sum a_i\otimes b_i$ be a solution of AYBE of weight~$\lambda$ on an associative algebra~$A$.
A linear map $P_r\colon A\to A$ defined by~\eqref{AYBE2RB} is an RB-operator of weight~$-\lambda$ on~$A$.
\end{sta}

\begin{thm}[\cite{AYBE-ext}]
The map $r\to P_r$ is the bijection between the set of the solutions of AYBE of weight~$\lambda$
on $M_n(F)$ and the set of RB-operators of weight $\lambda$ on $M_n(F)$.
\end{thm}

The analogues of classical or associative Yang---Baxter
equations for alternative and Jordan algebras were defined
in \cite{Goncharov,Zhelyabin}. The connection between solutions
of the Yang---Baxter equation and RB-operators on the Cayley---Dickson algebra $C(F)$
was found in~\cite{BGP}.

\section{Rota---Baxter operators of nonzero weight}

In~\S4, first, we give a list of algebras whose all RB-operators of nonzero weight
are either splitting (odd-dimensional simple Jordan algebras of bilinear form,
the simple Jordan Kaplansky superalgebra $K_3$, \S4.2) or triangular-splitting
($\textrm{sl}_2(\mathbb{C})$, the Witt algebra (for homogeneous RB-operators), \S4.1).

Second, we prove that all RB-operators of nonzero weight on a unital power-associative
algebra~$A$ over a~field of characteristic zero are splitting provided that
$A = F1\oplus N$ (as vector spaces), where $N$ is a~nil-algebra (Theorem~\ref{thm9}).
In Corollary~\ref{cor4} (\S4.4), we get that all RB-operators of nonzero weight on
the Grassmann algebra are splitting.

Third, we state the main result about RB-operators of nonzero weight on the matrix algebra.
Over an algebraically closed field~$F$ of characteristic zero,
given an RB-operator~$R$ of nonzero weight on $M_n(F)$,
there exists $\psi\in\Aut(M_n(F))$ such that $R^{(\psi)}(1)$ is diagonal (Theorem~\ref{thm11}, \S4.5).
As a corollary, we show that any RB-operator on $M_3(F)$
up to conjugation with an automorphism
preserves the subalgebra of diagonal matrices (Corollary~\ref{Cor:diagM3}, \S4.5).

\subsection{Triangular-splitting RB-operators}

Given an algebra~$A$ and a vector space~$B$,
we say that $B$~is a $A$-module, if the action
of~$A$ on~$B$ is defined, i.e., $ab,ba\in B$ for all $a\in A$, $b\in B$.

Consider a construction of RB-operators of nonzero weight which
generalizes Statement~\ref{sta3}.

\begin{sta}[\cite{Guil}]\label{sta6}
Let an algebra $A$ be a direct sum of subspaces $A_-,A_0,A_+$,
moreover, $A_\pm,A_0$ are subalgebras of $A$,
and $A_\pm$ are $A_0$-modules.
If $R_0$ is an RB-operator of weight~$\lambda$ on $A_0$,
then an operator $P$ defined as
\begin{equation}\label{RB:SubAlg2}
P(a_-+a_0+a_+) = R_0(a_0) - \lambda a_+,\quad
a_{\pm}\in A_{\pm},\ a_0\in A_0,
\end{equation}
is an RB-operator of weight~$\lambda$ on~$A$.
\end{sta}

Let us call an RB-operator of nonzero weight
defined by~(\ref{RB:SubAlg2}) as triangular-splitting one
provided that at least one of $A_-,A_+$ is nonzero.

If $A_0 = (0)$, then $P$ is splitting RB-operator on $A$.
If $A_0$ has trivial product, then every linear map on $A_0$ is suitable as $R_0$.

\begin{rem}
Let $P$ be a triangular-splitting RB-operator on an algebra $A$
with subalgebras $A_\pm,A_0$. Then the operator $\phi(P)$ is the
triangular-splitting RB-operator with the same subalgebras,
$$
\phi(P)(a_-+a_0+a_+) = -\lambda a_- + \phi(R_0)(a_0),\quad
a_{\pm}\in A_{\pm},\ a_0\in A_0.
$$
\end{rem}

\begin{exm}\label{exm9}
In \cite{Witt}, all homogeneous RB-operators on the Witt
algebra $W {=} \Span\{L_n \mid n\in\mathbb{Z}\}$
over $\mathbb{C}$ with the Lie product $[L_m,L_n] = (m-n)L_{m+n}$ were described.
A~homogeneous RB-operator with degree $k\in\mathbb{Z}$ satisfies the condition
$R(L_m)\in\Span\{L_{m+k}\}$ for all $m\in\mathbb{Z}$. Due to \cite{Witt},
all nonzero homogeneous RB-operators of weight~1 on~$W$ up to the action~$\phi$
and conjugation with automorphisms of~$W$ are the following:

(WT1) $R(L_m) = 0$, $m\geq -1$, $R(L_m) = -L_m$, $m\leq -2$;

(WT2) $R(L_m) = 0$, $m\geq 1$, $R(L_m) = -L_m$, $m\leq -1$,
$R(L_0) = kL_0$ for some $k\in \mathbb{C}$.
\end{exm}

The RB-operators (WT1) and (WT2) are triangular-splitting
with abe\-lian $A_0 = L_0$.

Let us consider the simple Lie algebra $\mathrm{sl}_2(\mathbb{C})$
with the Chevalley basis $e,f,h$.

\begin{thm}[\cite{KonovDissert,sl2}]\label{thm8}
All nontrivial RB-operators of nonzero weight on $\mathrm{sl}_2(\mathbb{C})$
up to conjugation with an automorphism are the following:

(a) the splitting RB-operator with
$A_1 = \Span\{e + \alpha h\}$,
$A_2 = \Span\{h,f\}$, $\alpha\neq0$;

(b) the triangular-splitting RB-operator with subalgebras
$A_- {=} \Span\{e\}$, $A_+ {=} \Span\{f\}$ and $A_0 = \Span\{h\}$.
\end{thm}

It is easy to calculate that Example~\ref{exm6} is
a~particular case of (b) (up to conjugation).

\subsection{The simple Jordan algebra of a bilinear form}

Let $J = J_{n+1}(f) = F1\oplus V$ be a direct vector-space sum of $F$
and finite-dimensional vector space~$V$, $\dim V = n > 1$,
and $f$ be a nondegenerate symmetric bilinear form on~$V$.
Under the product
\begin{equation}\label{FormProduct}
(\alpha\cdot1 + a)(\beta\cdot1 + b)
 = (\alpha\beta+f(a,b))\cdot1 + (\alpha b +\beta a),\quad
\alpha,\beta\in F,\ a,b\in V,
\end{equation}
the space $J$ is a simple Jordan algebra \cite{Nearly}.

\begin{thm}[\cite{BGP}]\label{thmJordanNonzero}
Let $J$ be an odd-dimensional simple Jordan algebra of bilinear form
over a field of characteristic not two.
Each RB-operator $R$ of nonzero weight on $A$ is splitting
and $R(1) = 0$ up to $\phi$.
\end{thm}

Let us choose a basis $e_1$, $e_2$, \ldots, $e_n$ of $V$ such that
the matrix of the form $f$ in this basis is diagonal with
elements $d_1,d_2,\ldots,d_n$ on the main diagonal.
As $f$ is nondegenerate, $d_i\neq0$ for each $i$.

\begin{exm}[\cite{BGP}]
Let $J_{2n}(f)$, $n\geq2$, be the simple Jordan algebra of bilinear form~$f$
over an algebraically closed field of characteristic not two.
Let $R$ be a linear operator on $J_{2n}(f)$
defined by a matrix $(r_{ij})_{i,j=0}^{2n-1}$
in the basis $1,e_1,e_2,\ldots,e_n$
with the following nonzero entries
\begin{gather*}
r_{00} = -3,\quad r_{01} = \sqrt{d_1},\quad r_{10}
 = -\frac{1}{\sqrt{d_1}},\quad r_{jj} = -1,\ j=1,\ldots,2n-1, \\
r_{i\,i+1} = \frac{d_{i+1}}{d_i}\sqrt{ -\frac{d_i}{d_{i+1}} }, \quad
r_{i+1\,i} = - \sqrt{ -\frac{d_i}{d_{i+1}} },\ i=2,\ldots,2n-2.
\end{gather*}
Then $R$ is a non-splitting RB-operator of weight~$2$ on $J_{2n}(f)$.
\end{exm}

The analogue of Theorem~\ref{thmJordanNonzero} holds
for the simple 3-dimensional Jordan superalgebra $K_3$
which is defined over a field of characteristic not two as follows:
$K_3 = A_0\oplus A_1$, $A_0 = \Span\{e\}$ (even part),
$A_1 = \Span\{x,y\}$ (odd part),
$$
e^2 = e, \quad
ex = xe = \frac{x}{2},\quad
ey = ye = \frac{y}{2},\quad
xy = -yx = \frac{e}{2},\quad
x^2 = y^2 = 0.
$$

\begin{thm}[\cite{BGP}]
All RB-operators of nonzero weight on the simple Jordan Kaplansky
superalgebra $K_3$ are splitting.
\end{thm}

\subsection{Sum of fields}

\begin{pro}[\cite{AnBai,Braga}]\label{pro1}
Let $A = Fe_1\oplus Fe_2\oplus\ldots\oplus Fe_n$
be the direct sum of copies of a~field~$F$ with $e_i e_j = \delta_{ij}e_i$.
A linear operator $R(e_i) = \sum\limits_{k=1}^n r_{ik}e_k$,
$r_{ik}\in F$, is an RB-operator of weight~1 on~$A$
if and only if the following conditions are satisfied:

(SF1) $r_{ii} = 0$ and $r_{ik}\in\{0,1\}$
or $r_{ii} = -1$ and $r_{ik}\in\{0,-1\}$ for all $k\neq i$;

(SF2) if $r_{ik} = r_{ki} = 0$ for $i\neq k$,
then $r_{il}r_{kl} = 0$ for all $l\not\in\{i,k\}$;

(SF3) if $r_{ik}\neq0$ for $i\neq k$,
then $r_{ki} = 0$ and
$r_{kl} = 0$ or $r_{il} = r_{ik}$ for all $l\not\in\{i,k\}$.
\end{pro}

\begin{proof}
The RB-identity~(\ref{RB}) is equivalent to the equalities
$$
r_{kl}(1+2r_{kk}-r_{kl}) = 0,\quad
r_{ik}r_{kl}+r_{ki}r_{il} = r_{il}r_{kl},\ i\neq k,
$$
from which Proposition~\ref{pro1} follows.
\end{proof}

\begin{exm}[\cite{Atkinson,Miller69}]\label{exm10}
The following operator is an RB-operator of weight~1 on~$A$:
$$
R(e_i) = \sum\limits_{l=i+1}^s e_l,\ 1\leq i<s,\quad
R(e_s) = 0,\quad
R(e_i) = -\sum\limits_{l=i}^{n} e_l,\ s+1\leq i\leq n.
$$
Note that $R(1) = 0e_1+1e_1+\ldots+(s-1)e_s-e_{s+1}-2e_{s+2}-\ldots-(n-s)e_n$.
\end{exm}

\begin{rem}
From (SF2) and (SF3) it easily follows that
$r_{ik}r_{ki} = 0$ for all $i\neq k$.
In~\cite{AnBai}, the statement of Proposition~\ref{pro1}
was formulated with this equality and (SF1)
but without (SF2) and (SF3).
That is why the formulation in \cite{AnBai} seems to be not complete.
\end{rem}

\begin{rem}
The sum of fields in Proposition~\ref{pro1} can be infinite.
\end{rem}

More about RB-operators of nonzero weight on a finite direct
sum of copies of a~field, including counting of all of them, splitting ones, etc.
see in~\cite{Gub2018}.

\subsection{Grassmann algebra}

\begin{lem}\label{lem3}
Let $A$ be a unital power-associative algebra over a field
of characteristic zero and let $R$ be an RB-operator
of nonzero weight on~$A$. If $R(1)$ is nilpotent, then $R(1) = 0$
and $R$ is splitting.
\end{lem}

\begin{proof}
The following formulas hold in all associative RB-algebras of weight~$\lambda$~\cite{GuoMonograph}:
\begin{gather}
n! R^n(1)  = \sum\limits_{k=1}^n (-1)^{n-k}\lambda^{n-k}s(n,k)(R(1))^k, \label{Stirling1}\\
(R(1))^n   = \sum\limits_{k=1}^n k!\lambda^{n-k} S(n,k)R^k(1), \label{Stirling2}
\end{gather}
where $s(n,k)$ and $S(n,k)$ are Stirling numbers
of the first and second kind respectively.
The proof of the formulas~(\ref{Stirling1}),~(\ref{Stirling2})
for power-associative algebras is absolutely the same as for
associative ones (see, e.g., \cite[Thm. 3.1.1]{GuoMonograph}).

Suppose that $t$ is maximal nonzero power of $R(1)$.
If $t = 0$, then we are done by Lemma~\ref{lem1}(a). Suppose that $t\geq1$.
Then all elements $R(1)$, $(R(1))^2$, $\ldots$, $(R(1))^t$ are linearly independent.
By~(\ref{Stirling1}) and the properties of Stirling numbers,
elements $R(1),R^2(1),\ldots$, $R^t(1)$ are also linearly independent.

The number $S(n,k)$ equals the number of ways of partitioning a set
of $n$ elements into $k$ non-empty subsets. It is well-known that
\begin{equation}\label{Stirling2Prop}
S(n,n) = 1, \quad
S(n,n-1) = \binom{n}{2}, \quad
S(n,n-2) = \binom{n}{3} + 3\binom{n}{4}.
\end{equation}

Without loss of generality, assume that $\lambda = 1$.
We apply~(\ref{Stirling2Prop}) to write down~(\ref{Stirling2}) for $t+1$ and $t+2$:
\begin{multline}\label{NilpR(1)t+1}
0 = (t+1)!R^{t+1}(1) + t! \binom{t+1}{2}R^t(1) \\
 + (t-1)!\left(\binom{t+1}{3} + 3\binom{t+1}{4}\right)R^{t-1}(1)
 + \ldots,
\end{multline}
\begin{multline}\label{NilpR(1)t+2}
0 = (t+2)!R^{t+2}(1) + (t+1)!\binom{t+2}{2} R^{t+1}(1) \\
 + t!\left(\binom{t+2}{3} + 3\binom{t+2}{4} \right)R^t(1)
 + \ldots\!.
\end{multline}

Act $(t+2)R$ on~(\ref{NilpR(1)t+1}) and subtract
the result from~(\ref{NilpR(1)t+2}) to get the equality
\begin{equation}\label{NilpR(1)Conseq}
0 = t!\binom{t+2}{2} R^{t+1}(1) + (t-1)!\left(\binom{t+2}{3}
  + 6\binom{t+2}{4} \right)R^t(1)+\ldots = 0.
\end{equation}
Let us multiply~(\ref{NilpR(1)t+1}) by $(t+2)/2$ and
subtract it from~(\ref{NilpR(1)Conseq}):
$$
0 = -\frac{1}{2}\binom{t+2}{3} R^t(1) + \sum\limits_{i=1}^{t-1}\alpha_i R^i(1),\quad \alpha_i\in F,
$$
a~contradiction with linear independence of the elements
$R(1),\ldots,R^t(1)$.
\end{proof}

\begin{cor}
Let $A$ be a unital power-associative algebra over
a~field of characteristic zero. Any RB-operator $R$ on $A$
of nonzero weight such that $\Imm(R)$ is a nil-algebra is splitting.
\end{cor}

Now, we are ready to state the first of the three main results of the paper.
We find sufficient conditions under which a unital algebra has only splitting RB-operators.

\begin{thm}\label{thm9}
Let $A$ be a unital power-associative algebra over a field
of characteristic zero and $A = F1\oplus N$ (as vector spaces),
where $N$ is a nil-algebra. Then each RB-operator~$R$ on~$A$ of
nonzero weight is splitting and (up to $\phi$) we have $R(1) = 0$.
\end{thm}

\begin{proof}
Let $R$ be an RB-operator of nonzero weight $\lambda$ on~$A$.

Suppose there exists $x\in A$ such that $R(x) = 1$
and let $\lambda = -1$. We have
\begin{equation}\label{PreimageOf1}
1 = R(x)R(x)
 = 2R(R(x)) - R(x^2) = 2 - R(x^2),
\end{equation}
therefore, $R(x^2) = 1$. Analogously, $R(x^k) = 1$ for all $k\in\mathbb{N}$.
Hence, $x = \alpha\cdot1+a$ for some $\alpha\in F^*$.
From $x-x^2\in \ker R$, we have $\alpha(1-\alpha)1+a(1-2\alpha-a)\in\ker R$.
If $\alpha_1 = \alpha(1-\alpha) = 0$, then $\alpha = 1$ and $\ker R$ contains
$a + a^2$, $a + a^2 - (a+a^2)^2 = a -2a^3-a^4$ and so on. Thus, $a\in \ker R$
and $R(x) = \alpha R(1) = 1$. By Lemma~\ref{lem1}(a), we are done.

Suppose that $\alpha_1\neq 0$. Consider
$$
x-x^3 = \alpha(1-\alpha^2)\cdot1 + a(1-3\alpha^2-3\alpha a-a^2)\in \ker R
$$
and denote $\alpha_2 = \alpha(1-\alpha^2)$. By the same reasons as above,
we can consider only the case $\alpha_2\neq0$. From
$$
(x-x^3) - (1+\alpha)(x-x^2)
 = \alpha_1 a + (1-2\alpha)a^2 - a^3 \in \ker R,
$$
we analogously obtain $a\in\ker R$.

Let us prove that $R(A)\subseteq N$.
Suppose there exists $x\in A$ such that $R(x) = 1 + a$ for $a\in N$.
As $\Imm(R)$ is a subalgebra, $1+2a+a^2\in\Imm(R)$
and hence $a+a^2\in\Imm(R)$. It remains to repeat the above arguments
to deduce that $a\in\Imm(R)$. So, $1\in\Imm(R)$, a~contradiction.

Applying Lemma~\ref{lem3}, we obtain that $R$ is splitting.
\end{proof}

\begin{cor}
Let $A$ be a unital power-associative algebra with no idempotents
except $0$ and $1$. Then $-\lambda\id$ is the only
invertible RB-operator of nonzero weight~$\lambda$ on~$A$.
\end{cor}

\begin{proof}
Suppose that $R$ is invertible RB-operator of weight~$-1$ on~$A$
(we rescale by Statement~\ref{sta1}(b).
For $x\in A$ such that $R(x) = 1$, we have by~(\ref{PreimageOf1})
that $x$ is an idempotent in~$A$. Thus, $x = 1$
and $R$ is splitting by Lemma~\ref{lem1}(a). As $R$ is invertible, $R = \id$.
\end{proof}

Let $\mathrm{Gr}_n$ ($\mathrm{Gr}_\infty$)
denote the Grassmann algebra of a vector space
$V = \Span\{e_1,\ldots,e_n\}$
($V = \Span\{e_i\mid i\in\mathbb{N}_+\}$).

In \cite{BGP}, it was proved that all
RB-operators of nonzero weight on $\mathrm{Gr}_2$ are splitting.
Now we extend this result as follows.

\begin{cor}\label{cor4}
Given an RB-operator $R$ of nonzero weight on $\mathrm{Gr}_\infty$ or
$\mathrm{Gr}_n$, we have $R$~is splitting and $R(1) = 0$ up to $\phi$.
\end{cor}

\begin{cor}
Given the polynomial algebra $F[x_1,\ldots,x_n]$,
let $A$ be its quotient by the ideal
$I = \langle x_1^{m_1},\ldots,x_n^{m_n}\rangle$, $m_i\in\mathbb{N}$.
Each RB-operator $R$ of nonzero weight on $A$ is splitting
and $R(1) = 0$ up to $\phi$.
\end{cor}

\subsection{Matrix algebra}

We need the following bounds on the dimensions
of $\ker(R)$ and $\Imm(R)$ for an RB-operator $R$ on $M_n(F)$.

\begin{lem}\label{lem4}
Let $F$ be either an algebraically closed field or
a field of characteristic zero
and $R$ be an RB-operator of nonzero weight on $M_n(F)$
which is not splitting.

(a) We have $n-1\leq \dim(\ker R)\leq n^2-n$.

(b) If $\dim(\ker R) = n-1$, then $\Imm(R)$ contains the identity matrix
and it is conjugate to the subalgebra $\Span\{e_{ij}\mid j>1$ or $i=j=1\}$.

\end{lem}

\begin{proof}
In \cite{Agore}, the bound $\dim A\leq n^2-n+1$
was proved for a proper maximal subalgebra~$A$ of $M_n(F)$ over
a field $F$ of characteristic zero. This bound can be proved
when $F$ is an algebraically closed field of any characteristic
since the required Wedderburn---Maltsev decomposition
over such fields exists~\cite[p. 143]{Behrens}.

(a) As $R$ is not splitting, $1\not\in\ker R$ by Lemma~\ref{lem1}(a).
Hence, $\dim\ker R\leq n^2-n$. The bound $n-1\leq \dim(\ker R)$
follows from $\dim(\Imm R)\leq n^2-n+1$.

(b) Let $\dim(\ker R) = n-1$.
Then $\dim(\Imm R) = n^2-n+1$ and we are done by~\cite{Agore}.
\end{proof}

Let $D_n$ denote the subalgebra of all diagonal matrices in $M_n(F)$
and $L_n$ ($U_n$) the set of all strictly lower (upper) triangular
matrices in $M_n(F)$.

\begin{exm}[\cite{BGP}]\label{exm11}
Decomposing $M_n(F) = L_n\oplus D_n\oplus U_n$ (as vector spaces),
we have a~triangular-splitting RB-operator
defined with $A_- = L_n$, $A_+ = U_n$, $A_0 = D_n$
or $A_- = U_n$, $A_+ = L_n$, $A_0 = D_n$.
All RB-operators of nonzero weight on $D_n$ were described in Proposition~\ref{pro1}.
\end{exm}

In some sense, all RB-operators of nonzero weight on $M_2(F)$ when $\char F\neq2$
were described in~\cite{BGP}. Let us refine thus result as follows.

\begin{thm}\label{RBOnM2nonzero}
Let $F$ be an algebraically closed field of characteristic zero.
Every nontrivial RB-operator of weight~1 on $M_2(F)$
(up to conjugation with an automorphism of $M_2(F)$, up to $\phi$ and transpose)
equals to one of the following cases:

(M1) $R\begin{pmatrix}
x_{11} & x_{12} \\
x_{21} & x_{22} \\
\end{pmatrix}
 = \begin{pmatrix}
0 & - x_{12} \\
0 & x_{11}
\end{pmatrix}$,

(M2) $R\begin{pmatrix}
x_{11} & x_{12} \\
x_{21} & x_{22} \\
\end{pmatrix}
 = \begin{pmatrix}
-x_{11} & - x_{12} \\
0 & 0
\end{pmatrix}$,

(M3) $R\begin{pmatrix}
x_{11} & x_{12} \\
x_{21} & x_{22} \\
\end{pmatrix}
 = -x_{21}\begin{pmatrix}
\alpha & \alpha\gamma \\
1 & \gamma
\end{pmatrix}$,\quad $\alpha,\gamma \in F$,

(M4) $R\begin{pmatrix}
x_{11} & x_{12} \\
x_{21} & x_{22} \\
\end{pmatrix}
 = \begin{pmatrix}
\alpha x_{12} & - x_{12}\\
-x_{21} & (1/\alpha)x_{21}
\end{pmatrix}$,\quad $\alpha\in F\setminus\{0\}$,

(M5) $R\begin{pmatrix}
x_{11} & x_{12} \\
x_{21} & x_{22} \\
\end{pmatrix}
 = \begin{pmatrix}
x_{22}-x_{11} + \alpha x_{21} & (-\alpha^2/4) x_{21} - x_{12}\\
0 & 0
\end{pmatrix}$,\quad $\alpha\in F$.
\end{thm}

\begin{proof}
Let $R$ be an RB-operator of weight~1 on $M_2(F)$.
In~\cite{BGP}, it was stated that
either $R$~is defined by Example~\ref{exm11}
or $R$~is splitting with one subalgebra being unital.
It is easy to check all non-splitting RB-operators
defined by Example~\ref{exm11} are exactly (M1) and (M2).

Now, consider the case when $R$ is splitting, i.e.,
$M_2(F) = \ker (R) \oplus \ker (R+\id)$ (as vector spaces).
We may assume that $\dim (\ker (R))\geq \dim (\ker (R+\id))$.

{\sc Case I}, $\dim(\ker (R)) = 3$.
By Lemma~\ref{lem4}(b), $\ker (R)$ is unital and conjugate to the subalgebra of upper-triangular matrices.
So, $\ker(R+\id) = Fv$ for
$v = \begin{pmatrix}
\alpha & \beta \\
1 & \gamma
\end{pmatrix}$.
Since $\ker (R+\id)$ is a~subalgebra in $M_2(F)$,
we have the condition $\beta = \alpha\gamma$ and it is (M3).

{\sc Case II}, $\dim(\ker (R)) = 2$.
Suppose that $\ker (R)$ is unital subalgebra.
So, $\ker(R+\id)$ consists of matrices of rank not greater than~1. By~\cite{Meshulam},
we may assume that $\ker(R+\id) = \Span\{e_{11},e_{12}\}$.
Thus, $\ker(R)$ consists of linear combinations of~1 and
$w = \begin{pmatrix}
\alpha & \beta \\
1 & 0 \end{pmatrix}$.
We have either $\ker(R)\cong F\oplus F$ or $\ker(R)$ has a one-dimensional radical.
Since $w^2 = \beta\cdot1 + \alpha w$, we have the second variant if and only if
$\beta = -\alpha^2/4$, it is (M5).

In the case $\ker(R)\cong F\oplus F$, we may suppose that $\ker(R)$ consists of diagonal matrices.
Denote $R(e_{12}) = (\delta_{ij})$ and $R(e_{21}) = (\gamma_{ij})$.
Since $\Imm(R') = \ker(R)$, we get $\delta_{12} = \gamma_{21} = -1$
and $\delta_{21} = \gamma_{12} = 0$.
Since $\ker(R') = \Imm(R)$ is a subalgebra of $M_2(F)$
not containing unit, the determinants of $R(e_{12})$ and $R(e_{21})$ are zero.
So, $\delta_{11}\delta_{22} = \gamma_{11}\gamma_{22} = 0$.

From
$$
R(e_{12})R(e_{21}) = \delta_{22}R(e_{21}) + \gamma_{22}R(e_{12}),
$$
we receive the relations
$$
\delta_{22}\gamma_{22} = 0,\quad
\delta_{11}\gamma_{22} + \delta_{22}\gamma_{11} = \delta_{11}\gamma_{11}+1.
$$
Considering
$$
R(e_{21})(e_{12}) = \delta_{11}R(e_{21}) + \gamma_{11}R(e_{12}),
$$
we get
$$
\delta_{11}\gamma_{11} = 0, \quad
\delta_{11}\gamma_{22} + \delta_{22}\gamma_{11} = \delta_{22}\gamma_{22}+1.
$$
From all these equations, we see that exactly two of
$\delta_{11},\delta_{22},\gamma_{11},\gamma_{22}$ are zero.
Up to transpose, we may assume that $\delta_{22} = \gamma_{11} = 0$.
Finally, $\delta_{11}\gamma_{22} = 1$ and it is (M4).
\end{proof}

\begin{rem}
In general case $M_n(F)$, not all of RB-operators of weight~1
are splitting or are defined by Example~\ref{exm11}.
For example, a linear map
$R\colon M_3(F)\to M_3(F)$ defined as follows:
$R(e_{13}) = -e_{13}$, $R(e_{23}) = -e_{23}$,
$R(e_{33}) = e_{22}$, $R(e_{kl}) = 0$ for all other matrix unities
$e_{kl}$, is such an RB-operator of weight~1.
However, $R$ is triangular-splitting with subalgebras
$A_- = \Span\{e_{31},e_{32}\}$,
$A_+ = \Span\{e_{13},e_{23}\}$,
$A_0 = \Span\{e_{11},e_{12},e_{21},e_{22},e_{33}\}$.
\end{rem}

We need some preliminary results before proving the main theorem of the current section.

Let $F_n(m) = \sum\limits_{j=1}^{m}j^n$ for natural $n,m$.
It is well known that $F_1(m) = m(m+1)/2$,
$F_2(m) = m(m+1)(2m+1)/6$, $F_3(m) = (F_1(m))^2$.
For any $n$,
\begin{equation}\label{sum-formula}
F_n(m) = \frac{1}{n+1}\sum\limits_{j=0}^n(-1)^j
 \binom{n+1}{j} B_j m^{n+1-j},
\end{equation}
where $B_0=1,B_1,\ldots,B_n$ are Bernoulli numbers.

\begin{lem}[\cite{Miller66,Ogievetsky}]\label{lem5}
Let $A$ be a unital power-associative algebra, $R$ be an RB-operator
of weight $-1$ on~$A$, $a = R(1)$.
Then $R(a^n) = F_n(a)$ for all $n\in\mathbb{N}$.
\end{lem}

\begin{proof}
For $n = 0$, we have $R(a^0) = R(1) = a$ by the definition. Suppose that $n>0$.
With the help of~(\ref{Stirling1}) and (\ref{Stirling2}), we calculate
\begin{multline}\label{GKP-proof}\allowdisplaybreaks
R(a^n)
 = \sum\limits_{k=1}^n k!(-1)^{n-k}S(n,k)R^{k+1}(1) \\
 = \sum\limits_{k=1}^n \frac{k!}{(k+1)!}(-1)^{n-k}S(n,k)
   \sum\limits_{t=1}^{k+1}s(k+1,t)a^t \\
 = \sum\limits_{j=1}^{n+1} (-1)^{n+1-j}a^j\left(
 \sum\limits_{k\geq1}\frac{(-1)^{k+1-j}}{k+1}S(n,k)s(k+1,j)\right).
\end{multline}

Applying the equality (6.99) from \cite{GKP}
$$
\sum\limits_{k\geq1}\frac{(-1)^{k+1-j}}{k+1}S(n,k)s(k+1,j)
 = \frac{1}{n+1}\binom{n+1}{j} B_{n+1-j},
$$
the formulas~(\ref{sum-formula}) and~(\ref{GKP-proof}) coincide.
\end{proof}

Let us formulate the useful result about the generalized Vandermonde determinant.

\begin{lem}[{\cite[Thm.~20]{Krattenhaler}}]\label{lem6}
Let $t$ be a positive integer and let $A_t(X)$ denote the $t\times m$ matrix
$$
\!\begin{pmatrix}
1 & 0 & 0 & \dots & 0\\
X & 1 & 0 & \dots & 0\\
X^2 & 2X & 2 & \dots & 0\\
X^3 & 3X^2 & 6X & \dots & 0\\
\hdotsfor5\\
X^{t-1} & (t-1)X^{t-2} & (t-1)(t-2)X^{t-3} & \dots  & (t-1)\cdots(t-m+1)X^{t-m}
\end{pmatrix}\!
$$
where the first column consists of powers of $X$ and then every consequent column is formed by derivation of the previous one with respect to~$X$.
Given a partition $t = m_1 + \ldots + m_k$, there holds
\begin{equation} \label{eq:FlHa1}
\det(A_{m_1}(X_1)\,A_{m_2}(X_2)\dots A_{m_k}(X_k))
 = \left(\prod _{i=1}^{k}\prod _{j=1}^{m_i-1}j!\right)
\prod _{1\le i<j\le k}(X_j-X_i)^{m_i m_j}.
\end{equation}
\end{lem}

\begin{lem}\label{lem7}
Given a unital algebra~$A$ and an RB-operator $R$
of nonzero weight~$\lambda$~on~$A$,

(a) at least one of $\ker (R),\ker(R+\lambda\id)$ is nonzero,

(b) if $A$ is simple and $R$ is invertible, then $R = -\lambda\id$.
\end{lem}

\begin{proof}
Fix $\lambda = 1$.
Denote the space $A$ under the product
$[x,y] = R(x)y + xR(y) + xy$ as~$A'$.
It is easy to verify that $R,R+\id$ are
homomorphisms from $A'$ to $A$.

(a) Suppose that $\ker (R) = \ker(R+\id) = (0)$, so, $R$ and $R+\id$
are invertible and we may consider $\psi = (R+\id)R^{-1}\in\Aut(A)$.
Thus, $0 = \psi(1) - 1 = (R+\id)R^{-1}(1) - 1 = R^{-1}(1)$,
a contradiction.

(b) If $R$ is invertible, then $R$ is an isomorphism from $A'$ to $A$.
Applying (a), we have that $\ker(R+\id)\neq(0)$.
As $\ker(R+\id)$ is a nonzero ideal in $A'\cong A$ and $A$ is simple,
thus, $A' = \ker(R+\id)$ or $R = -\id$ on the entire space $A$.
\end{proof}

Now we are ready to prove the main result about
RB-operators of nonzero weight on the matrix algebra.

\begin{thm}\label{thm11}
Given an algebraically closed field $F$ of characteristic zero
and an RB-operator $R$ of nonzero weight~$\lambda$ on $M_n(F)$,
there exists a $\psi\in\Aut(M_n(F))$ such that
the matrix $R^{(\psi)}(1)$ is diagonal and (up to $\phi$)
the set of its diagonal elements has a form
\linebreak
$\{-p\lambda,(-p+1)\lambda,\ldots,-\lambda,0,\lambda,\ldots,q\lambda\}$
with $p,q\in\mathbb{N}$.
\end{thm}

\begin{proof}
If $R(1)$ is nilpotent, we are done by Lemma~\ref{lem3}.
Otherwise, consider the algebra $A$ generated by $1$ and $R(1)$.
If $1$ and $R(1)$ are linearly dependent, we are done by Lemma~\ref{lem1}(a).
From~(\ref{Stirling1}),~(\ref{Stirling2}) it follows
that $A$ is closed under the action of~$R$.
By Lemma~\ref{lem7}(a), we may suppose that $\ker(R)\neq(0)$ on~$A$.

Up to conjugation, we suppose that $R(1)$
is in the Jordan form~$J$.
Let $m(x) = (x-\lambda_1)^{m_1}\ldots(x-\lambda_k)^{m_k}$
be a minimal polynomial of $J$, $\lambda_i\neq\lambda_j$ for $i\neq j$.
Express $J$ as a sum of Jordan blocks
$J(\lambda_1),\ldots,J(\lambda_k)\in M_n(F)$.
Introduce matrices
$e_{i}^s\in M_n(F)$, $i=1,\ldots,k$,
such that $J(\lambda_i) = \lambda_i e_{i}^1 + e_{i}^2$,
$e_i^1$ is a~diagonal matrix in $M_n(F)$ and
$e_{i}^s = \big(J(\lambda_i) - \lambda_i e_{i}^1\big)^{s-1}$
for all $s\geq 2$. Note that $e_i^s = 0$ for $s>m_i$.

Lemma~\ref{lem6} implies that
$\big\{e_{i}^1,\ldots,e_{i}^{m_i}\mid i=1,\ldots,k\big\}$
is a linear basis of $A$. Moreover, one of $\lambda_i$ is zero
(denote $\lambda_k = 0$), otherwise $R$ is invertible on $A$.

Another linear basis of $A$ is $1,a,a^2,\ldots,a^{m-1}$, where $m = m_1+\ldots+m_k$.
By Lemma~\ref{lem5}, $R(a^k) = F_k(a)$. Hence,
$R(1),R(a),R(a^2),\ldots,R(a^{m-2})$ are linearly independent.
So, $\dim(\ker R) = 1$ on $A$.

From Lemma~\ref{lem6}, all $e_{i}^s$, $s=1,\ldots,m_i$, lie in $\Imm(R)$
for $i=1,\ldots,k-1$. Also, it follows from $R(1)\in\Imm(R)$  that
$e_{k}^2,\ldots,e_{k}^{m_k}\in\Imm(R)$.
Therefore, $R(v)$ has zero projection on $e_{k}^1$ for all $v\in A$
in the decomposition on the basis $e_i^s$, $i=1,\ldots,k$, $s=1,\ldots,m_i$.

For the ideal $A_k = \Span\big\{e_k^s\mid s=1,\ldots,m_k\big\}$,
consider the induced RB-operator $R_k$ from $R$ by Lemma~\ref{lem2}.
We have that $R_k\big(e_{k}^1\big)$ is nilpotent. As $e_k^1$ is a unit in $A_k$,
by Lemma~\ref{lem3} we have $R_k\big(e_k^1\big) = 0$.
Representing $A$ as $A'\oplus A_k$
for $A' = \Span\big\{e_{i}^s\mid i=1,\ldots,k-1;s=1,\ldots,m_i\big\}$,
we define
\begin{equation}\label{Diagonal:formulas}
R(1) = a + b,\quad
R(e') = a_1 + b_1,\quad
R\big(e_k^1\big) = a_2,
\end{equation}
where $a,a_1,a_2\in A'$, $b,b_1\in A_k$,
$e' = \sum\limits_{i=1}^{k-1}e_i^1$.
As $e' + e_k^1 = 1$, we get $b_1 = b = e_k^2$, $a_1 + a_2 = a$.

Calculating $R\big(e_k^1\big)R\big(e_k^1\big)$ by~(\ref{RB}), we have
$a_2(a_2+1) = 0$, i.e., $a_2 = -\sum\limits_{i\in I}e_i^1$
for some $I\subset\{1,\ldots,k-1\}$.
Introduce $\lambda_i' = \lambda_i$, if $i\not\in I$,
and $\lambda_i' = \lambda_i+1$, otherwise.

Let $x$ be a nonzero vector from $\ker(R)$.
Calculating $R(x)R(1)$ by~(\ref{RB}),
we get $xR(1)\in \ker R$.
As $\ker R$ is one-dimensional algebra,
$x^2 = \alpha x$ for $\alpha\in F$.
Further, we will consider three possibilities,
Case~1: $\alpha = 0$ and $x^2 = 0$. If $\alpha \neq 0$, then
$x$ is diagonal, i.e., $x = x'+x_k$ for
$x'\in\Span\big\{e_i^1\mid i=1,\ldots,k-1\big\}$ and
$x_k\in\Span\big\{e_k^1\big\}$.
Since $xR(1)\in\ker R$, either
$e_k^1\in\ker R$ (Case~2) or
$e_i^1\in\ker R$ for some $i\in\{1,\ldots,k-1\}$ (Case~3).

{\bf Case 1}: $x^2 = 0$, then by $xR(1)\in\ker R$
we have $x = \gamma e_{i_1}^{m_{i_1}}$ for some $i_1\in\{1,\ldots,k-1\}$,
$m_i>1$ and nonzero $\gamma\in F$. Let $\gamma = 1$ and denote $x$ as $x_1$.
As it was noted above, there exists $x_2\in A$ such that $R(x_2) = x_1$.
From
$$
0 = x_1^2 = R(x_2)R(x_2) = R\big(x_1x_2 + x_2x_1 - x_2^2\big)
  = -R\big(x_2^2\big),
$$
we have $x_2^2\in\ker(R)$. Representing $x_2$ as $x_2' + x_2''$ for
$x_2'\in\Span\big\{e_i^s\mid s=1,\ldots,m_i\big\}$
and $x_2''\in\Span\big\{e_j^s\mid j\neq i,s=1,\ldots,m_j\big\}$,
we have $(x_2')^2 = \varepsilon x_1$, $\varepsilon\in F$,
and $(x_2'')^2 = 0$.

Also, by
$$
\lambda_i x_1
 = R(1)x _1
 = R(1)R(x_2)
 = R(R(1)x_2 + x_1 - x_2) = R(R(1)x_2) - x_1,
$$
we get $R(1)x_2 = (\lambda_i+1)x_2 + \delta x_1$, $\delta\in F$.
Thus, $-x_2' + e_{i_1}^2x_2' = \delta x_1$ and
$(\lambda_i+1)x_2'' = R(1)x_2''$. From these equalities,
we get $x_2' \in \Span\{x_1\}$ and $x_2'' = \beta e_{i_2}^{m_{i_2}}$
for $\beta\neq0$ and $i_2$ such that
$\lambda_{i_2} = \lambda_{i_1}+1$, $m_{i_2}>1$. Moreover, $x_2^2 = 0$.

On the step $t$, we find $x_t$ such that $R(x_t) = x_{t-1}$. From
\begin{gather*}
0 = x_{t-1}^2 = R(x_t)R(x_t) = 2R\big(x_t x_{t-1}- x_t^2\big),\\
R(x_{t-1}+x_tR(1)-x_t) = R(1)R(x_t) = R(1)x_{t-1}
 = (\lambda_{i_1}+t-2)x_{t-1} + u,
\end{gather*}
where $u\in\Span\{x_1,\ldots,x_{t-2}\}$, by induction, we get
\begin{gather}
x_t^2 = 2x_t x_{t-1} + \psi x_1,\quad \psi\in F, \label{DiagMat:Nilp1} \\
R(1)x_t = (\lambda_{i_1}+t-1)x_t + w,\quad w\in\Span\{x_1,\ldots,x_{t-1}\}. \label{DiagMat:Nilp2}
\end{gather}

Represent $x_t$ as $\sum\limits_{j=1}^{t-1}x_{t,j} + x_t^*$
for $x_{t,j}\in\Span\big\{e_{i_j}^s\mid s=1,\ldots,m_{i_j}\big\}$,
$x_t^*\in\Span\big\{e_p^s\mid p\not\in\{i_1,\ldots,i_{t-1}\},s\ge1\big\}$.
By~(\ref{DiagMat:Nilp1}) and~(\ref{DiagMat:Nilp2}),
we get $\big(x_t^*\big)^2 = 0$ and
$R(1)x_t^* = (\lambda_{i_1}+t-1)x_t^*$.
Thus, either $x_t^* = 0$ or
$x_t^* = \chi e_{i_t}^{m_{i_t}}$, $\chi\in F\setminus\{0\}$,
$\lambda_{i_t} = \lambda_{i_1}+t-1$ and $m_{i_t}>1$.

By~(\ref{DiagMat:Nilp2}), we have
$(j-t)x_{t,j} + e_{i_j}^2 x_{t,j}\in\Span\big\{e_{i_j}^{m_{i_j}}\big\}$,
$1\leq j\leq t-1$. So, $x_{t,j}\in\Span\big\{e_{i_j}^{m_{i_j}}\big\}$.
Hence, $x_t^*\neq0$ and $x_t^2 = 0$.

We can continue the process endless, as all $e_i^{m_i}$ for $m_i>1$ lie in $\Imm (R)$.
It is a~contradiction to the fact that $A$ is a finite-dimensional algebra.

{\bf Case 2}: $e_k^1\in\ker(R)$.
From $e_k^1R(1) = e_k^1(a+b) = b\in\ker R$,
we get that $b = 0$, it means $m_k = 1$.
Define $R'$ as induced RB-operator $R$ on $A'$ (see Lemma~\ref{lem2}).
So, $R'(e') = a$ and by Lemma~\ref{lem6}, $R'$ is invertible on $A'$.
We finish Case~2 by the following result:

\begin{lem}\label{lem8}
Given an invertible RB-operator $R'$ on $A'$,
we have $m_i = 1$ for all $i = 1,\ldots,k-1$
and $\{\lambda_1,\ldots,\lambda_{k-1}\} = \{1,2,\ldots,k-1\}$.
\end{lem}

\begin{proof}
Define $x_k$ such that $(R')^k(x_k) = 1$. We have that $R'(x_1) = e'$.
From~(\ref{PreimageOf1}), we get $x_1 = x_1^2$.
From $a_1 = R'(x_1)R'(e') = a_1 + R'(x_1a_1-x_1)$,
we have $x_1a_1 = x_1$.
The only possibility is the following:
one of $\lambda_i$ equals~1, $m_i = 1$ and $x_1 = e_i^1$.
Exchange indices $1,\ldots,k-1$
in such way that $\lambda_1' = \lambda_1 = 1$.

For $x_2$, we have
$$
a_1 x_1 = R'(x_2)R'(1)
 = R'(x_1 + x_2a_1) - R'(x_2)
 = R'(x_1 + x_2a_1) - x_1,
$$
so, $R'(x_1 + x_2a_1) = x_1(a_1+1) = 2x_1 = 2R'(x_2)$, i.e.,
\begin{equation}\label{Diagonal:PreimageOf1}
x_1 + x_2a_1 = 2x_2.
\end{equation}
Representing $x_2$ as $\alpha x_1 + x_1^*$,
$x_1^*\in\Span\big\{e_i^s\mid 2\leq i\leq k-1;s=1,\ldots,m_i\big\}$,
from~(\ref{Diagonal:PreimageOf1}) we have
$\alpha = 1$ and either $x_2 = x_1$ or
one of $\lambda_i'$
(denote $\lambda_2'$) equals~2,
$m_2 = 1$ and $x_2 = x_1 + \beta e_2^1$
for some $\beta\in \mathbb{C}^*$.
Note that $\lambda_2' = \lambda_2$,
otherwise $2 = \lambda_2' = \lambda_2 +1$
or $\lambda_2 = \lambda_1 = 1$, a contradiction.
If $x_2 = x_1$, then $R'(x_2) = x_1 = R(x_1) = e'$,
so $a'$ is diagonal, $k=2$ and $m_1 = 1$.
If $x_2\neq x_1$, we proceed with $x_3$ an so on.
In every step $t\geq3$ we deal with the equality
$$
a_1 x_{t-1} = R'(x_t)R'(1)
 = R'(x_{t-1} + x_t a_1) - R'(x_t)
 = x_{t-2} + R'(x_t a_1) - x_{t-1}.
$$
Express $x_t = y_t + z_t$ for
$y_t\in\Span\big\{e_1^1,\ldots,e_{t-1}^1\big\}$
and $z_t\in\Span\{e_i^s\mid i\geq t; s=1,\ldots,m_i\}$.
Then $a_1 z_t = tz_t$ and so either $z_t = 0$ or
one of $\lambda_i'$
(denote $\lambda_t'$) equals~$t$,
$m_t = 1$ and $x_t = y_1 + \gamma e_t^1$
for some $\gamma\in \mathbb{C}^*$.

The space $A'$ is finite-dimensional,
so $x_p = \sum\limits_{j=1}^{p-1}\alpha_j x_j$
for some $p$. Consider the minimal such~$p$.
So,
$$
e' = (R')^p(x_p)
  = \sum\limits_{j=1}^{p-1}\alpha_j (R')^{p-j}(e').
$$
By~(\ref{Stirling1}), we get
$e'\in\Span\big\{a_1,a_1^2,\ldots,a_1^{p-1}\big\}$.
It may happen only if $p-1$ is a degree
of the minimal polynomial of $a_1$,
the spectrum of $a_1$ coincide with $\{1,2,\ldots,p-1\}$
and $m_1 = \ldots = m_{p-1} = 1$.
We have also proved that $\lambda_i' = \lambda_i$ for all $i=1,\ldots,k-1$.
Lemma~\ref{lem8} is proved.
\end{proof}

{\bf Case 3}: $y_1 = e_{i_1}^1\in\ker(R)$ for some $i_1\in\{1,\ldots,k-1\}$.
From $e_{i_1}^1R(1)\in \ker R$, we get $m_{i_1} = 1$.

Consider $y_2\in A$ such that $R(y_2) = y_1 = e_{i_1}^1$
and $y_2\in\Span\big\{e_i^s\mid i\neq i_1\big\}$. From
\begin{gather*}
y_1 = R(y_2)R(y_2) = 2R(y_1 y_2) - R\big(y_2^2\big), \\
\lambda_{i_1} y_1 = R(y_2)R(1) = R(y_2R(1)) - y_1,
\end{gather*}
we have
$y_2 + y_2^2 = 0$ and $(\lambda_{i_1}+1)y_2 = y_2 R(1)$.
We conclude that $y_2 = -e_{i_2}^1$ for some $i_2$.
Moreover, $m_{i_2} = 1$ and $\lambda_{i_2} = \lambda_{i_1}+1$.

On the step $t$, we find
$y_t\in \Span\big\{e_i^s\mid i\neq i_1\big\}$
such that $R(y_t) = y_{t-1}$. From
\begin{gather*}
y_{t-1}^2 = R(y_t)R(y_t) = 2R(y_{t-1} y_t) - R\big(y_t^2\big), \\
y_{t-1} R(1) = R(y_t)R(1) = R(y_{t-1} + y_tR(1)) - y_{t-1},
\end{gather*}
by induction arguments, we get that
$y_t^2 + (-1)^{t}y_t$ and $y_tR(1) - (\lambda_{i_1}+t-1)y_t$
lie in $\Span\{y_1,\ldots,y_{t-1}\}$.
So, $y_t = (-1)^{t+1}e_{i_t}^1 + w$ for $w\in \Span\{y_1,\ldots,y_{t-1}\}$,
and $m_{i_t} = 1$, $\lambda_{i_t} = \lambda_{i_1}+t-1$.

As $A$ is finite-dimensional and $e_{k}^1$ is the only
element from $e_1^1,\ldots,e_k^1$
which is not in $\Imm(R)$, we get $y_p = (-1)^{p-1}e_{k}^1 + u$,
$u\in\Span\{y_1,\ldots,y_{p-1}\}\subseteq A'$, on a step $p$.

Consider
$A_1 = \Span\big\{e_{i_1}^1,\ldots,e_{i_p}^1\big\}$ and
$A_2 = \Span\big\{e_i^s\mid i\not\in\{i_1,\ldots,i_p\};s=1,\ldots,m_i\big\}$.
We have $A = A_1 \oplus A_2$.
Moreover, $A_1$ is closed under the action of $R$ and $\dim (\ker R) = 1$ on $A_1$.
Thus, $A_2$ is closed under $R$ and $R$ is invertible
on $A_2$. We are done by Lemma~\ref{lem8}.
\end{proof}

\begin{rem}
For every $p,q\in\mathbb{N}$,
we apply Examples~\ref{exm10} and~\ref{exm11}
to construct an RB-operator $R$ on a matrix algebra with diagonal matrix $R(1)$ which has the numbers
$-p\lambda,(-p+1)\lambda,\ldots,-\lambda,0,\lambda,\ldots,q\lambda$
on the diagonal.
\end{rem}

Let us call an RB-operator $R$ on $M_n(F)$ {\it diagonal}, if
$R^{(\psi)}(D_n)\subseteq D_n$ for some $\psi\in\Aut(M_n(F))$.
By Theorem~\ref{RBOnM2nonzero}, all RB-operators of nonzero weight on $M_2(F)$ are diagonal.

Let us apply Theorem~\ref{thm11} to prove the following general statement.

\begin{pro}\label{pro2}
Let $R$ be an RB-operator of nonzero weight on $M_n(F)$
and $R(1) = (a_{ij})_{i,j=1}^n$ be a diagonal matrix with
$k$~different values of diagonal elements:
$a_{ii} = \lambda_j$, if $m_{j-1}+1\leq i\leq m_j$,
where $m_1+m_2+\ldots+m_k = n$ and $m_0 = 0$.
Define $B = B_1 \oplus \ldots \oplus B_k$ (as vector spaces) for
$B_j = \Span\{e_{ii}\mid m_{j-1}+1\leq i\leq m_j\}$.
Then $B$ is $R$-invariant subalgebra of $M_n(F)$.
If $k = n$, then $R$ is a diagonal RB-operator.
\end{pro}

\begin{proof}
Without loss of generality, assume that $R$ is an RB-operator of weight $-1$. By~(\ref{RB}),
$$
R(1)R(x) = R(R(1)x + R(x) - x),\quad
R(x)R(1) = R(xR(1) + R(x) - x)
$$
and $[R(1),R(x)] = R([R(1),x])$.

Considering $x = e_{ij}$ for $e_{ij}\in B_k$,
we get $R(1)R(e_{ij}) = R(e_{ij})R(1)$.
From the last equality we have $R(e_{ij})\in B$.

If $k = n$, then $B$ is the subalgebra of diagonal matrices
and $R$ is diagonal RB-operator by the definition.
\end{proof}

\begin{cor}\label{Cor:diagM3}
Let $F$ be an algebraically closed field of characteristic zero.
All RB-operators of nonzero weight on $M_3(F)$ are diagonal.
\end{cor}

\begin{proof}
Consider an RB-operator $R$ on $M_3(F)$.
By Theorem~\ref{thm11}, we may suppose that $R(1)$ is a diagonal matrix.
If $R(1)$ is scalar, then $R$ is splitting with one of subalgebras
being unital by Lemma~\ref{lem1}(a),
this case we will consider later.
If all diagonal elements of $R(1)$ are pairwise distinct,
then $R$ is diagonal by Proposition~\ref{pro2}.
Thus, $R(1)$ has exactly two different values on the diagonal.
Applying $\phi$, one of the values is zero.
By the proof of Theorem~\ref{thm11}, we have three cases:

1. $R(e_{11}+e_{22}) = e_{11}+e_{22}$, $R(e_{33}) = 0$,
$R(1) = 1(e_{11} + e_{22}) + 0e_{33}$;

2. $R(e_{11}+e_{22}) = 0$, $R(e_{33}) = -e_{11}-e_{22}$,
$R(1) = -1(e_{11} + e_{22}) + 0e_{33}$;

3. $R(e_{11}+e_{22}) = -e_{33}$, $R(e_{33}) = 0$,
$R(1) = 0(e_{11} + e_{22}) - 1e_{33}$.

Denote $B_1 = \Span\{e_{ij}\mid 1\leq i,j\leq 2\}$,
$B_2 = \Span\{e_{33}\}$, $B = B_1\oplus B_2$.
By Proposition~\ref{pro2}, $B$ is an $R$-invariant subalgebra of $M_3(F)$.
Consider the induced RB-operator $R_1$ from $R$ on $B_1\cong M_2(F)$.
By~Theorem~\ref{RBOnM2nonzero}, $R_1$ is diagonal on $B_1$
up to conjugation with an automorphism $\psi$ of $B_1$.
Extend the action of $\psi$ on the entire algebra $B$ as follows:
$\psi (e_{33}) = e_{33}$.
It is easy to calculate directly that
$R^{\psi}(D_3)\subseteq D_3$. Let us show it for the case~1,
the proofs for the cases~2 and~3 are analogous. Indeed,
$R_1^{\psi}(e_{11}) = d_1$, $R_1^{\psi}(e_{22}) = d_2$
for $d_1,d_2\in \Span\{e_{11},e_{22}\}$. Thus,
$R^{\psi}(e_{11}) = d_1 + \alpha e_{33}$,
$R^{\psi}(e_{22}) = d_2 + \beta e_{33}$
for some $\alpha,\beta\in F$.
Finally, $R^{\psi}(e_{33}) = \psi^{-1}R\psi(e_{33})
= \psi^{-1}R(e_{33}) = 0$.

Now, let $R$ be a splitting RB-operator of weight~1 with one of subalgebras being unital,
i.e., $M_3(F)$ equals a direct vector-space sum of subalgebras
$A_1 = \ker(R)$ and $A_2 = \ker (R+\id) = \Imm (R)$,
and we may assume that $1\in A_2$.
By Lemma~\ref{lem4}(a), we have $2\leq \dim (A_1)\leq 6$.
Consider all these variants:

1. $\dim A_1 = 2$, $\dim A_2 = 7$.
By Lemma~\ref{lem4}(b), up to conjugation $D_3\subset A_2$, we are done.

2. $\dim A_1 = 3$, $\dim A_2 = 6$.
Define $B$ as a~linear span of all matrix unities $e_{ij}$
except $e_{21}$ and $e_{31}$. As the algebra $B$ is a
unique up to conjugation maximal subalgebra in $M_3(F)$~\cite{Agore},
we may assume that $A_2$ is a subalgebra in $B$. Let
\begin{equation}\label{M2}
M = \Span\{e_{22},e_{23},e_{32},e_{33}\}\cong M_2(F),
\end{equation}
then $3\leq \dim(A_2\cap M)\leq 4$.
If $A_2\cap M = M$, then $e_{11}\in A_2$ (as $1\in A_2$)
and $\alpha e_{12}+\beta e_{13}\in A_2$ for some $\alpha,\beta \in F$
which are not zero simultaneously.
Hence, $(\alpha e_{12}+\beta e_{13})e_{22} = \alpha e_{12}\in A_2$
and $\beta e_{13}\in A_2$. If $\alpha = 0$, then $e_{13}\in A_2$
as well as $e_{13}e_{32} = e_{12}$, a contradiction.
If $\beta = 0$, then $e_{12}e_{23} = e_{13}\in A_2$,
a contradiction.

Thus, $\dim(A_2\cap M) = 3$.
By \cite{Agore}, there exists a matrix $S\in M$ such that
$S^{-1}(A_2\cap M)S = \Span\{e_{22},e_{23},e_{33}\}$.
Let $T = e_{11} + S$, then $B$ is invariant under
the conjugation with $T$. So, we get that
$A_2 = D_3\oplus U_3$ (as vector spaces) and therefore
$D_3$ is up to conjugation $R$-invariant.

3. $\dim A_1 = 4$, $\dim A_2 = 5$.
Consider $A_2\cap M$ for $M$ defined by~(\ref{M2}).
As above, we assume that $A_2$ is a subalgebra of $B$.
By the dimensional reasons, $\dim(A_2\cap M)\geq2$.
If $A_2\cap M = M$, then $D_3\subset A_2$
and we are done. If $\dim(A_2\cap M) = 3$,
then we deal with it as in the variant~2 and get $D_3\subset A_2$.
So, $\dim(A_2\cap M) = 2$ and as vector spaces
$$
A_2 = \Span\{e_{1i}+d_i\mid i=1,2,3;d_i\in M\}\oplus (A_2\cap M).
$$
Up to conjugation with the matrix from $B$,
we have either $A_2\cap M = \Span\{e_{22},e_{33}\}$
or $A_2\cap M = \Span\{e_{22}+e_{33},e_{23}\}$,
or $A_2\cap M = \Span\{e_{22},e_{23}\}$.
In the first case $D_3\subset A_2$ and we have done.
In the second case, $d_1 = 0$ as $1\in A_2$.
Also, $e_{12},e_{13}\in A_2$ as
$(e_{22}+e_{33})(e_{12}+d_2) = d_2\in A_2$
and analogously $d_3\in A_2$.
So, $A_2 = U_3\oplus\Span\{e_{11},e_{22}+e_{33}\}$ (as vector spaces).
Let us show that the third case either
can be reduced to the similar subalgebra or contains $D_3$ up to conjugation.

Indeed, we may assume that $d_1,d_2,d_3\in\Span\{e_{32},e_{33}\}$.
As $1\in A_2$, we have $e_{11}+e_{33}\in A_2$.
Also, $(e_{12}+d_2)e_{22} = e_{12} + d_2 e_{22}\in A_2$.
Thus, $d_2 = ke_{32}$ for some $k\in F$.
Since $(e_{12}+d_2)e_{23} = e_{13} + ke_{33}\in A_2$,
we have $A_2 = \Span\{e_{11}+e_{33},e_{12}+ke_{32},e_{13}+ke_{33},e_{22},e_{23}\}$.
For $k\neq 0$, we apply the conjugation with the matrix
$P = \begin{pmatrix}
1 & 0 & 1 \\
0 & 1 & 0 \\
k & 0 & 0 \\
\end{pmatrix}$
and we get $P^{-1}A_2P = D_3\oplus \Span\{e_{12},e_{21}\}$ (as vector spaces),
so $A_2$ up to conjugation contains $D_3$.
For $k = 0$, we get $A_2 = U_3\oplus \Span\{e_{11}+e_{33},e_{22}\}$ (as vector spaces),
the algebra conjugated to the one obtained in the second case.
Let us proceed with this variant of $A_2$.

As $\dim A_1 = 4$ and $A_1\cap A_2 = (0)$,
we have $\dim(A_1\cap (D_3\oplus U_3)) = 1$ (as vector spaces).
Let $t$ be a nonzero vector from $A_1\cap (D_3\oplus U_3)$ (as vector spaces).
Without loss of generality, we may assume that
$t = e_{11}+\gamma e_{22}+\delta e_{33} + d$,
where $\delta\neq1$ and $d = xe_{12}+ye_{13}+ze_{23}\in U_3$.
Due to the arguments stated above, $t^2 = t$, $\delta = 0$ and
we have two possibilities:
$t = e_{11} + xe_{12} + ye_{13}$ or
$t = e_{11} + ye_{13} + e_{22} + ze_{23}$.
In the first case, we apply the conjugation with
the matrix $U = \begin{pmatrix}
1 & -x & -y \\
0 & 1 & 0 \\
0 & 0 & 1 \\
\end{pmatrix}$
and we get $e_{11}\in U^{-1}A_1U$,
$\dim(D_3\cap U^{-1}A_2U) = 2$, we are done.
In the second case we conjugate with the matrix
$V = \begin{pmatrix}
1 & 0 & -y \\
0 & 1 & -z \\
0 & 0 & 1 \\
\end{pmatrix}$, so, $e_{11}+e_{22}\in V^{-1}A_1V$
and $\dim(D_3\cap V^{-1}A_2V) = 2$.
Thus, we finish this variant.

4. $\dim A_1 = 5$, $\dim A_2 = 4$.
Consider the algebra $C = A_1\oplus F1$ (as vector spaces),
it is 6-dimensional unital subalgebra in $M_3(F)$.
Due to the variant 2, we may assume that $C = D_3\oplus U_3$ (as vector spaces).
So, $D_3$ is up to conjugation $R$-invariant.

5. $\dim A_1 = 6$, $\dim A_2 = 3$.
Consider the algebra $C = A_1\oplus F1$ (as vector spaces),
it is 7-dimensional unital subalgebra in $M_3(F)$.
By Lemma~\ref{lem4}(b), we are done.
\end{proof}

\begin{prob}
Does there exist not diagonal RB-operator on $M_n(F)$?
\end{prob}

\begin{prob}
To classify all diagonal RB-operators on $M_n(F)$.
\end{prob}

\subsection{Derivations of nonzero weight}

Given an algebra $A$ and $\lambda\in F$,
a linear operator $d\colon A\rightarrow A$
is called a derivation of weight~$\lambda$~\cite{DifOp}
if $d$ satisfies the identity
\begin{equation}\label{Diff}
d(xy) = d(x)y + xd(y) + \lambda d(x)d(y),\quad x,y \in A.
\end{equation}

Let us call zero operator and $-(1/\lambda)\id$ (if $\lambda\neq0$)
as trivial derivations of weight $\lambda$.

\begin{sta}[\cite{preLie}]\label{sta7}
Given an algebra $A$ and invertible derivation $d$ of weight~$\lambda$ on~$A$,
the operator $d^{-1}$ is an RB-operator of weight~$\lambda$ on~$A$.
\end{sta}

\begin{cor}
There are no nontrivial invertible derivations of nonzero weight on

(a) \cite{BGP} Kaplansky superalgebra $K_3$,

(b) Grassmann algebra, $\mathrm{sl}_2(\mathbb{C})$ and any unital simple algebra.
\end{cor}

\begin{proof}
(b) It follows from Corollary~\ref{cor4}, Theorem~\ref{thm8} and Lemma~\ref{lem7}(b) respectively.
\end{proof}

\section{Rota---Baxter operators of weight zero}

In \S5, we study RB-operators of weight zero.
In Lemma~\ref{lem9} (\S5.1), we give a list of general constructions
of RB-operators of weight zero which cover most of known examples of RB-operators.

In \S5.2, we state the general property of RB-operators of weight zero on unital algebras.
Given a unital associative (alternative, Jordan) algebraic algebra~$A$
over a field of characteristic zero,
there exists $N$ such that $R^N = 0$
for all RB-operators $R$ of weight zero on $A$ (Theorem~\ref{thm12}, \S5.2).

So, we define a Rota---Baxter index (RB-index) $\rb(A)$ on~$A$ as the minimal
natural number with such nilpotency property.
We prove that $\rb(M_n(F)) = 2n-1$ over a field~$F$
of characteristic zero (Theorem~\ref{thmRBMn}, \S5.4) and study RB-index
for the Grassmann algebra (\S5.3),
unital composition algebras and simple Jordan algebras (\S5.5).

\subsection{Constructions of Rota---Baxter operators}

Below, we say that an algebra $D$ is abelian if its product is trivial, i.e., zero.

\begin{lem}\label{lem9}
Given an algebra $A$, described below linear operator $R$ on $A$
is an RB-operator of weight zero,

(a) If $A = B\oplus C$ (as vector spaces) with abelian $\Imm R$,
$R\colon B\to C$, $R\colon C\to (0)$, $BR(B),R(B)B\subseteq C$. In particular,

(a1) If $A$ is a $\mathbb{Z}_2$-graded algebra $A = A_0\oplus A_1$
with abelian $\Imm R$ and $R\colon A_0\to A_1$, $R\colon A_1\to (0)$;

(a2) If $A$ is a $\mathbb{Z}_2$-graded algebra $A = A_0\oplus A_1$
with abelian odd part and $R\colon A_0\to A_1$, $R\colon A_1\to (0)$.

(b) If $A = B\oplus C$ (as vector spaces),
$R(B)C,CR(B)\subseteq C$, $R\colon B\to C$, $R\colon C\to (0)$, and
$R(b)R(b') = R(R(b)b' + bR(b'))$ for all $b,b'\in B$. In particular,

(b1) If $A$ is a $\mathbb{Z}_2$-graded algebra $A = A_0\oplus A_1$,
$R\colon A_1\to A_0$, $R\colon A_0\to (0)$,
$R(c)R(d) = R(R(c)d + cR(d)) = 0$ for all $c,d\in A_1$;

(b2) If $A$ is a $\mathbb{Z}_2$-graded algebra $A = A_0\oplus A_1$
with abelian even part, $R\colon A_1\to A_0$, $R\colon A_0\to (0)$, and
$R(c)d + cR(d) = 0$ for all $c,d\in A_1$.

(c) If $A = B\oplus C$ (as vector spaces),
$R\colon B\to B$, $R\colon C\to (0)$, $B^2 = (0)$,
$BC,CB\subseteq \ker R$. In particular,

(c1) If $A$ is a $\mathbb{Z}_2$-graded algebra $A = A_0\oplus A_1$
with abelian even part and $R\colon A_0\to A_0$, $R\colon A_1\to (0)$;

(c2) If $A$ is a $\mathbb{Z}_2$-graded algebra $A = A_0\oplus A_1$
with abelian odd part, $R\colon A_0\to (0)$, $R\colon A_1\to A_1$, and
$A_0\cdot\Imm R, \Imm R\cdot A_0\subseteq \ker R$.

(d) If $A$ is a $\mathbb{Z}_3$-graded algebra $A = A_0\oplus A_1\oplus A_2$
with $A_2^2 = (0)$, and $R\colon A_1\to A_2$, $R\colon A_0\oplus A_2\to (0)$.

(e) If $A = B\oplus C\oplus D$ (as vector spaces)
with $C^2 = D^2 = (0)$, $CD,DC\subseteq D$,
$R\colon B\to C$, $R\colon C\to D$, $R\colon D\to (0)$, moreover,
$R(b)b' + bR(b') = 0$, $R(bd) = R(b)d$, $R(db) = dR(b)$
for all $b,b'\in B$, $d\in D$. In particular,

(e1) If $A$ is a $\mathbb{Z}_3$-graded algebra $A = A_0\oplus A_1\oplus A_2$
with $A_0^2 = A_1^2 = (0)$,
$R\colon A_2\to A_0$, $R\colon A_0\to A_1$, $R\colon A_1\to (0)$,
moreover,
$R(a_2)b_2 + a_2R(b_2) = 0$, $R(a_2c_1) = R(a_2)c_1$,
$R(c_1a_2) = c_1R(a_2)$ for all $a_2,b_2\in A_2$, $c_1\in A_1$.

(f) If $A$ is a $\mathbb{Z}_3$-graded algebra $A = A_0\oplus A_1\oplus A_2$
with $A_1A_2, A_2A_1, A_2^2 = (0)$,
$R\colon A_0\to A_1$, $R\colon A_1\to A_2$, $R\colon A_2\to (0)$,
moreover, $R(c)R(d) = R(R(c)d+cR(d))$ for all $c,d\in A_0$.

(g) If $A = B\oplus C\oplus D$ (as vector spaces)
with $CD,DC,D^2\subset D$, abelian $R(C)$,
$R\colon B\to C$, $R\colon C\to D$, $R\colon D\to (0)$, moreover,
$R(x)R(y) = R(R(x)y + xR(y))$
for all $x\in B$, $y\in B\oplus C$ or $x\in B\oplus C$, $y\in B$.

(h) \cite{Gub2017} If $A = B \oplus C$ (as vector spaces), where $B$ is a subalgebra
with an RB-operator~$P$ of weight zero, $BC,CB\subseteq \ker P\oplus C$,
and $R$ is defined as follows:
$R|_{B} \equiv P$, $R|_{C} \equiv 0$.

(i) \cite{preLie}
If $A = A_1\oplus A_2$, $R_i$ are RB-operators on $A_i$, $i=1,2$,
and the linear map $R\colon A \to A$ is defined by the formula
$R(x_1+x_2) = R_1(x_1)+R_2(x_2)$, $x_1\in A_1$, $x_2\in A_2$.
\end{lem}

\begin{proof}
Straightforward.
\end{proof}

Let $A$ be an algebra.
Note that conjugation with automorphisms of~$A$ and scalar multiple
does not change the correspondence of an RB-operator of weight zero on $A$ to Lemma~\ref{lem9}.

\begin{exm}
By Example~\ref{exm1}, a linear map $l_e$ in an associative algebra $A$ with $e^2 = 0$
is an RB-operator of weight zero. If $e\in Z(A)$, then $l_e$ corresponds to Lemma~\ref{lem9}(a1)
for $A_0 = (1-e)A$ and $A_1 = eA$.
\end{exm}

\begin{exm}[\cite{Aguiar00-2}]
Given an associative algebra $A$ and an element $e$ such that $e^2 = 0$,
a linear map $l_e r_e$ which acts on $A$ as $x\to exe$ is an RB-operator of weight zero on~$A$.
It corresponds to Lemma~\ref{lem9}(a) for $C = eA + Ae$ and
any $B$ such that $A = B \oplus C$ (as vector spaces).
\end{exm}

\begin{exm}
In \cite{Embedding}, given a~variety $\Var$ of algebras and a pre-$\Var$-algebra~$A$,
an enveloping RB-algebra~$B$ of weight zero and the variety $\Var$ for~$A$ was constructed.
By the construction, $B = A\oplus A'$ (as vector spaces),
where $A'$ is a copy of $A$ as a vector space,
and the RB-operator $R$ was defined as follows: $R(a')= a$, $R(a)=0$, $a\in A$.
By the definition of the operations on~$B$~\cite{Embedding},
the operator~$R$ corresponds to Lemma~\ref{lem9}(a2).
\end{exm}

\begin{exm}
In \cite{Witt}, all homogeneous RB-operators on the Witt algebra $W$
over $\mathbb{C}$ were described (see Example~\ref{exm9} about the definitions).
According to \cite{Witt}, all nonzero homogeneous RB-operators of weight zero on~$W$
with degree $k$ up to the multiplication
on $\alpha\in\mathbb{C}$ are the following:

(W1) $R(L_m) = 0$, $m\neq 0$, $R(L_0) = L_0$;

(W2) $R(L_m) = \delta_{m+2k,0}L_{m+k}$, $k\neq0$, $m\in\mathbb{Z}$;

(W3) $R(L_m) = (\delta_{m+2k,0}+2\delta_{m+3k,0})L_{m+k}$,
 $k\neq0$, $m\in\mathbb{Z}$;

(W4) $R(L_m) = \frac{k}{m+2k}\delta_{m+k,l\mathbb{Z}}L_{m+k}$,
 $l$ doesn't divide $k$, $m\in\mathbb{Z}$.
\end{exm}

Note that the RB-operator (W1) corresponds to Lemma~\ref{lem9}(c),
the RB-operator (W2) --- by Lemma~\ref{lem9}(a), (W3) --- by Lemma~\ref{lem9}(b1) for
$A_1 = \Span\{L_{-3k},L_{-2k}\}$,
$A_0 = \Span\{L_m\mid m\not\in \{-3k,-2k\}\}$.
Finally, the RB-operator (W4) corresponds to Lemma~\ref{lem9}(b) for
$B = \Span\{L_m\mid m+k\in l\mathbb{Z}\}$ and
$C = \Span\{L_m\mid m+k\not\in l\mathbb{Z}\}$.

\begin{exm}
In \cite{preLie},
a classification of RB-operators on all 2- and some
3-dimensional pre-Lie algebras over $\mathbb{C}$ was given.
In particular, for the simple pre-Lie algebra
$S2 = \Span\{ e_1,e_2$,
$e_3 \mid e_1 e_2 = e_2 e_1 = e_3, e_3 e_1 = e_1, e_3 e_2 = -e_2 \}$,
all RB-operators of weight zero on $S2$ are of the form
$R(e_3) = \alpha e_3$, $R(e_1) = R(e_2) = 0$.
Thus, all RB-operators of weight zero on $S2$ correspond to Lemma~\ref{lem9}(c1) for
$A_0 = \Span\{e_3\}$ and $A_1 = \Span\{e_1,e_2\}$.
\end{exm}

\begin{exm}
In \cite{preLie}, the RB-operator $R$ on a class of simple pre-Lie algebras
$I_n = \Span\{ e_1,\ldots,e_n\mid e_n e_n = 2e_n,
e_n e_j = e_j, e_j e_j = e_n,1\leq j\leq n-1\}$ introduced in~\cite{Burde}
for $n\geq3$ was constructed as follows
$R(e_1) = e_n + \dfrac{1}{\sqrt{2-n}}\sum\limits_{i=2}^{n-1}e_i$,
$R(e_i)=0$, $2\leq i\leq n$.
It is easy to see that $R$ corresponds to Lemma~\ref{lem9}(b1) for
$A_0 = \Span\{e_2,\ldots,e_n\}$ and $A_1 = \Span\{e_1\}$.
\end{exm}

\begin{exm}\label{exm18}
In \cite{BGP}, it was proved that all RB-operators of weight zero
on the simple Kaplansky Jordan superalgebra $\mathrm{K}_3$
up to conjugation with $\Aut(\mathrm{K}_3)$ are the following:
$R(e) = R(x) = 0$, $R(y) = ae + bx$ for $a,b\in F$.
Thus, given an RB-operator $R$ of weight zero on $\mathrm{K}_3$,
we have $R^2 = 0$ and $R$ corresponds to Lemma~\ref{lem9}(b1)
for $A_0 = \Span\{e,x\}$ and $A_1 = \Span\{y\}$.
\end{exm}

In \cite{sl2-0}, all RB-operators of weight zero
on $\textrm{sl}_2(\mathbb{C})$ were described as 22 series.
In \cite{Kolesnikov} a~very simple classification of them was given;
up to conjugation with an automorphism of $\textrm{sl}_2(\mathbb{C})$
and scalar multiple all nonzero RB-operators
on $\textrm{sl}_2(\mathbb{C})$ are the following:

(L1) $R(e) = 0$, $R(f) = te-h$, $R(h)=2e$, $t\in \mathbb{C}$;

(L2) $R(e) = 0$, $R(f) = 2te+h$, $R(h)=2e+\frac{1}{t}h$, $t\in \mathbb{C}^*$;

(L3) $R(h) = h$, $R(e) = R(f) = 0$;

(L4) $R(f) = h$, $R(e) = R(h) = 0$;

(L5) $R(f) = e$, $R(e) = R(h) = 0$.

Note that Example~\ref{exm5} up to scalar multiple is (L1) for $t = 0$.

\begin{sta}\label{sta11}
All RB-operators on the simple Lie algebra $\textrm{sl}_2(\mathbb{C})$
correspond to Lemma~\ref{lem9}.
\end{sta}

\begin{proof}
It is easy to check that (L1) corresponds to Lemma~\ref{lem9}(e)
with $A_0 = \Span\{f\}$, $A_1 = \Span\{te-h\}$, $A_2 = \Span\{e\}$.
Cases (L2) and (L3) come from Lemma~\ref{lem9}(c1).
RB-operator $R$ of the type (L4) is constructed by Lemma~\ref{lem9}(h)
for $B = \Span\{f,h\}$ and $C = \Span\{e\}$, where
$R$ on~$B$ corresponds to Lemma~\ref{lem9}(b2).
Finally, (L5) corresponds to Lemma~\ref{lem9}(d).
\end{proof}

\subsection{Unital algebras}

\begin{lem}\label{lem10}
Let $A$ be a unital power-associative algebraic algebra and $R$ be an RB-operator
on $A$ of weight zero. Then $R(1)$ is nilpotent.
\end{lem}

\begin{proof}
In the case $\char F = p>0$, we have $(R(1))^p = p!R^p (1) = 0$.

Suppose that $\char F = 0$. Denote $a = R(1)$. Assume that $a$~is not nilpotent.
As $A$ is algebraic, we can consider a minimal $k$ such that
\begin{equation}\label{UnitPowAs}
\alpha_1 a+\alpha_2 a^2+\ldots +\alpha_k a^k = 0
\end{equation}
and not all of $\alpha_i$ are zero. Thus, $k>1$ and $\alpha_k\neq 0$.

By Lemma~\ref{lem1}(d), $a^m = m!R^m(1)$ for all $m\in\mathbb{N}$.
Hence, $R(a^m) = m!R^{m+1}(1) $
$= \frac{a^{m+1}}{m+1}$.
Let us apply the operator $(k+1)R-l_a$ to the left-hand side of~(\ref{UnitPowAs}):
\begin{multline}\label{UnitPowAsSubtr}
\left(\alpha_1 \frac{(k+1)a^2}{2} + \alpha_2\frac{(k+1)a^3}{3}
 + \ldots + \alpha_{k-1}\frac{(k+1)a^k}{k} + \alpha_k a^{k+1}\right) \\
 - (\alpha_1 a^2 + \alpha_2 a^3 + \ldots +\alpha_k a^{k+1}) \\
 = \alpha_1\left(\frac{k+1}{2}-1\right)a^2
 + \alpha_2\left(\frac{k+1}{3}-1\right)a^3 + \ldots
 + \alpha_{k-1}\left(\frac{k+1}{k}-1\right)a^k
 = 0.
\end{multline}
If $\alpha_{k-1} = 0$, then minimality $k$ implies
$\alpha_i = 0$ for all $i=1,\ldots,k-1$ and $a^k = 0$.
If $\alpha_{k-1}\neq0$, then the expressions~(\ref{UnitPowAs})
and~(\ref{UnitPowAsSubtr}) are proportional by a~nonzero scalar.
Hence, we obtain $\alpha_1 = 0$, $\alpha_2 = 0$, \ldots,
$\alpha_{k-1} = 0$, a~contradiction.
\end{proof}

\begin{lem}\label{lem11}
Let $A$ be a unital power-associative algebra over a field
of characteristic zero and let $R$ be an RB-operator
of weight zero on $A$. Suppose that $(R(1))^m = 0$. Then

(a) $R^{2m} = 0$ if $A$ is associative or alternative,

(b) $R^{3m-1} = 0$ if $A$ is Jordan.
\end{lem}

\begin{proof}
(a) Suppose $A$ is an associative algebra.
By Lemma~\ref{lem1}(d), $R^m(1) = 0$.
By standard computations in RB-algebra
(see, e.g., \cite{BokutChen}), we write
\begin{multline}\label{UnitAs}
0 = R(x)R(1)R^{m-1}(1)
  = R^2(x)R^{m-1}(1) + R(xR(1))R^{m-1}(1) \\
  = \chi R^{m+1}(x) + \sum\limits_{l=1}^{m-1}\beta_l R^{m+1-l}(xR^l(1))
\end{multline}
for $x\in A$ and for some $\chi,\beta_l\in\mathbb{N}_{>0}$. Let $t \geq 2m$.
Multiple substitution of~(\ref{UnitAs}) into itself along with the equality
\begin{equation}\label{UnitAs2}
R^k(1)R^l(1) = \frac{1}{k!}(R(1))^k \frac{1}{l!}(R(1))^l
 = \frac{1}{k!l!}(R(1))^{k+l}
 = \frac{(k+l)!}{k!l!}R^{k+l}(1)
\end{equation}
leads to
\begin{multline}\label{UnitAs3}
R^t(x)
 = -\frac{1}{\chi}\sum\limits_{l=1}^{m-1}\beta_l R^{t-l}(xR^l(1))
 = \sum\limits_{2\leq l_1+l_2\leq m-1}\beta_{l_1,l_2} R^{t-l_1-l_2}(xR^{l_1+l_2}(1)) \\
 = \ldots = \sum\limits_{m-1\leq l_1+\ldots+l_{m-1}\leq m-1}\beta_{l_1,\ldots,l_{m-1}}
 R^{t-l_1-\ldots-l_m}(xR^{l_1+\ldots+l_m}(1)) \\
 = \beta_{1,\ldots,1}R^{t-m+1}(xR^{m-1}(1))
 = -\frac{m\beta_1\beta_{1,\ldots,1}}{\chi}R^{t-m}(xR^m(1)) + \ldots = 0.
\end{multline}
For alternative algebras, the proof can be repeated without any changes.
Indeed, any two elements in an alternative algebra generate an associative subalgebra.
Hence, we can avoid bracketings under the action of $R$
in~(\ref{UnitAs})--(\ref{UnitAs3}).

b) In the Jordan case, the computations similar to~(\ref{UnitAs})--(\ref{UnitAs3}) show that for $t\geq m+1$
\begin{equation}\label{UnitJord}
R^t(x)
 = -\frac{1}{\chi}\sum\limits_{l=1}^{t-1}\sum\limits_{p_1+\ldots +p_s = l}
 \beta_{l,p_1,\ldots,p_s} R^{t-l}(\ldots ((xy^{p_1})y^{p_2})\ldots y^{p_s} ),\quad y = R(1),
\end{equation}
where we have some nonassociative words in the alphabet $\{x,y\}$
under the action of $R$ in the right-hand side.

We know that $y^m = 0$. Let us show that any word $w(x,y)$ of the length $2m$ with
a~single occurrence of $x$ and $2m-1$ occurrences of $y$ equals zero in~$A$.
Indeed, by the Shirshov theorem~\cite[p. 71]{Nearly},
the subalgebra~$S$ of~$A$ generated by $\{x,y\}$ is a special one.
It means that there exists an associative enveloping algebra $E$
in which $S$ embeds injectively. Calculating the meaning of $w$ in~$E$,
we get a linear combination $L$ of associative words of the length $2m$
with a single occurrence of $x$ and $2m-1$ occurrences of $y$.
By the Dirichlet's principle, any associative word from $L$
has a subword $y^l$ with $l\geq m$. As $y^m = 0$ in $S$,
we have done. Therefore, to get zero it is enough to execute~\eqref{UnitJord} with $t = 3m-1$.
\end{proof}

\begin{thm}\label{thm12}
Let $A$ be a unital associative (alternative, Jordan)
algebraic algebra over a~field of characteristic zero.
Then there exists $N$ such that $R^N = 0$
for every RB-operator~$R$ of weight zero on $A$.
\end{thm}

\begin{proof}
Each associative, alternative or Jordan algebra is power-associa\-tive,
thus we may apply Lemma~\ref{lem10} for algebras from all of these varieties.
We are done by Lemma~\ref{lem11}.
\end{proof}

Let $A$ be an algebra (or just a ring).
Define Rota---Baxter index (RB-index) $\rb(A)$ of $A$ as follows
$$
\rb(A) = \min\{n\in\mathbb{N}\mid R^n = 0\mbox{ for an RB-operator }R \mbox{ of weight zero on }A\}.
$$
If such number is undefined, put $\rb(A) = \infty$.

\begin{lem}\label{lem12}
Let $A$ be a commutative associative (alternative) algebra and $e$ be a~nonzero idempotent of $A$.
For an RB-operator $R$ on $A$ of weight zero, $e\not \in \Imm(R)$.
\end{lem}

\begin{proof}
Suppose that $(0\neq)e\in\Imm (R)$, i.e., $e = R(x)$ for some $x\in A$.
Then $R(x) = R(x)R(x) = 2R(R(x)x) = 2R(ex)$.
On the one hand, $x - 2ex = k\in\ker R$. On the other hand,
\begin{multline*}
R(x) = 2R(ex) = 2R(e(2ex+k)) = 4R(ex)+2R(ek) \\
 = 2R(x) + 2R(x)R(k) - 2R(xR(k)) = 2R(x).
\end{multline*}
Therefore, $e = R(x) = 0$, a contradiction.
\end{proof}

\begin{rem}
The analogous statement is not valid for
associative and Jordan algebras, see Examples~\ref{exm7} and~\ref{exm18}.
\end{rem}

\begin{cor}\label{cor8}
Given a field $F$ of characteristic zero, let $A$ be an algebra over $F$
isomorphic to a direct (not necessary finite) sum of finite extensions
of a field $F$, in particular $A$ can be a sum of copies of $F$.
Then there are no nonzero RB-operators of weight zero on~$A$.
\end{cor}

\begin{proof}
Suppose that $R$ is a nontrivial RB-operator of weight zero on~$A$,
then $\Imm(R)$ contains a nonzero element~$a$.
Since $A$ has no nilpotent elements, the subalgebra $S$ of $\Imm(R)$ generated by~$a$
is a semisimple finite-dimensional commutative algebra and $S$~contains an idempotent~\cite[Thm. 1.3.2]{Herstein}.
It is a contradiction to Lemma~\ref{lem12}.
\end{proof}

\begin{rem}
Let us explain why there are only trivial RB-operators of weight zero on some semigroup
algebras of order~2 and~3 in~\cite{Semigroup}.

We recall some facts about semigroups due to~\cite{SemigroupMonograph}.
A set $A$ with a binary operation~$\cdot$ is called semigroup if the product $\cdot$ is associative.
An element $a\in S$ is called regular if $axa = a$ for some $x\in S$.
By an inverse semigroup we mean a semigroup in which every element is regular
and any two idempotents of $S$ commute with each other.
The Mashke's Theorem for semigroups~\cite[Thm. 5.26]{SemigroupMonograph} says:
the semigroup algebra $F[S]$ of a finite inverse semigroup $S$ is semisimple
if and only if the characteristic of $F$ is zero or a prime not dividing
the order of any subgroup of $S$.

In~\cite{Semigroup}, there are only trivial RB-operators of weight zero on semigroup algebras
of order~2 and~3 exactly for commutative inverse semigroups.
In this case, by the Mashke's Theorem for semigroups, $F[S]$ is a commutative semisimple algebra.
By the Wedderburn---Artin theorem~\cite[Thm. 2.1.7]{Herstein},
$F[S]$ is a direct sum of finite extensions of a field $F$, and all RB-operators of weight zero on $F[S]$
are trivial by Corollary~\ref{cor8}. In particular, it implies the affirmative answer
on the question from~\cite{Semigroup}: whether all RB-operators of weight zero are trivial
on $F[G]$ over a field $F$ of characteristic zero for any (prime order) cyclic group $G$.
Moreover, all RB-operators of weight zero are trivial on $F[G]$ for every finite abelian group $G$
if $\char F = 0$ or $\char F>|G|$.
\end{rem}

Given an associative algebra $A$, let us denote its Jacobson radical by $\Rad(A)$.
In an associative Artinian algebra $A$,
$\Rad(A)$ is the largest nilpotent ideal in~$A$ \cite[Thm. 1.3.1]{Herstein}.
The same is true for the radical $\Rad(A)$
of an alternative Artinian algebra~$A$ \cite[p. 250]{Nearly}.

\begin{thm}\label{thm13}
Let $A$ be a commutative associative (alternative)
finite-dimen\-sional algebra over a field of characteristic zero
and $R$ be an RB-operator on~$A$. Then

(a) $\Imm R\subseteq \Rad(A)$,

(b) $\rb(A)\leq 2m-1$ if $A$ is unital,
where $m$ is the index of nilpotency of $\Rad(A)$.
\end{thm}

\begin{proof}
(a) Suppose $\Imm R\not\subseteq \Rad(A)$,
then $\Imm R$ is not nilpotent algebra.
In this case, $\Imm R$ has an idempotent~\cite[Thm. 1.3.2]{Herstein}
(in alternative case, apply \cite[p.250]{Nearly}
and the standard lifting of an idempotent),
a~contradiction to Lemma~\ref{lem12}. So $\Imm R\subseteq \Rad(A)$.

(b) By (a), $R(1)\in \Rad(A)$.
The remaining proof is analogous to the proof of Lemma~\ref{lem11}(a)
with one change: by (a), we can write down~(\ref{UnitAs})
with $R(x)(R(1))^{m-1}$ instead of $R(x)(R(1))^m$.
Thus, it is enough to consider the power $t\geq 2m-1$.
\end{proof}

It is well-known that in characteristic zero,
the solvable radical of a finite-dimen\-sional Lie algebra
is preserved by any derivation \cite[p. 51]{Jacobson-Lie}
as well as the locally nilpotent radical and the nil-radical of
an (not necessary Lie or associative) algebra \cite{Slin'ko}.

\begin{cor}
Let $A$ be a commutative associative (alternative)
finite-di\-mensional algebra over a field of characteristic zero
and let $R$ be an RB-operator on $A$. Then $\Rad(A)$ is $R$-invariant.
\end{cor}

\begin{thm}\label{thm14}
Let $A$ be a unital associative algebra equal $F1\oplus N$
(as vector spaces), where $N$ is nilpotent ideal of index~$m$,
$\char F = 0$, and let $R$ be an RB-operator of weight zero on $A$. Then we have

(a) $\Imm R\subseteq N$,

(b) $\rb(A)\leq 2m-1$.
\end{thm}

\begin{proof}
(a) For $x = \alpha\cdot1+n\in \Imm R$, $n\in N$, the following is true:
$(x-\alpha\cdot1)^m = 0$. As $\Imm R$ is a subalgebra in~$A$,
$\alpha^m\cdot1\in \Imm R$. By Lemma~\ref{lem1}(b), $\alpha = 0$.

(b) Analogous to the proof of Theorem~\ref{thm13}.
\end{proof}

\begin{lem}\label{lem13}
Let $A$ be an algebra and let $e$ be an idempotent of $A$.
Then for an RB-operator~$R$ on $A$ of weight zero,
$e\not \in \Imm (R^k)\cap \ker R$ for $k\geq2$.
\end{lem}

\begin{proof}
Suppose that $e\in\Imm (R^2)\cap \ker R$, i.e.,
$e = R(x)$ for some $x\in \Imm R$.
Then $e = R(x) = R(x)R(x) = R(R(x)x+xR(x)) = R(ex+xe) = 0$,
as $\ker R$ is an $(\Imm R)$-module. It is a contradiction.
The proof for $k>2$ is analogous.
\end{proof}

\begin{cor}\label{cor10}
Let $A$ be an associative (alternative) finite-dimensional algebra
and let $R$ be an RB-operator of weight zero on~$A$. Then
$\Imm (R^k)\cap \ker R$ is a nilpotent ideal in $\Imm (R^k)$ for $k\geq2$.
\end{cor}

\begin{proof}
It follows from Lemma~\ref{lem13} and the stated above property of alternative algebras:
every finite-dimensional alternative not nilpotent algebra contains an idempotent~\cite{Nearly},
as $\Imm (R^k)\cap \ker R\subset \Rad(A)$ and $\Rad(A)$ is nilpotent.
\end{proof}

\subsection{Grassmann algebra}

Recall that $\mathrm{Gr}_n$ denotes the Grassmann algebra of the space
$V = \Span\{e_1,\ldots,e_n\}$ and let $A_0(n)$ be its subalgebra generated by $V$.

\begin{lem}\label{lem14}
Let $R$ be an RB-operator of weight zero on $\mathrm{Gr}_n$ and let $\char F = 0$.
Then we have
(a) $R(\mathrm{Gr}_n)\subseteq A_0(n)$;
(b) $R(e_1\wedge e_2\wedge\ldots \wedge e_n) = 0$;
(c) $(R(1))^{[(n+1)/2]+1} = 0$;
(d)~$\rb(\mathrm{Gr}_n)\leq 2\big[\frac{n+1}{2}\big] + 2$.
\end{lem}

\begin{proof}
(a) It follows from Theorem~\ref{thm14}(a).

(b) If $R(e_1\wedge e_2\wedge\ldots\wedge e_n)\neq 0$,
then there exists $x\in\mathrm{Gr}_n$ such that
$R(e_1\wedge\ldots\wedge e_n)x = e_1\wedge\ldots\wedge e_n$. Further,
$$
R(e_1\wedge\ldots\wedge e_n)R(x)
  = R( R(e_1\wedge\ldots\wedge e_n)x + e_1\wedge\ldots\wedge e_n R(x) )
  = R(e_1\wedge\ldots\wedge e_n),
$$
a contradiction, as $R(x)\in A_0(n)$.

(c) The linear basis of $A_0(n)$ consists of the vectors
$e_\alpha = e_{\alpha_1}\wedge e_{\alpha_2}\wedge\ldots\wedge e_{\alpha_s}$
for $\alpha = \{\alpha_1,\alpha_2,\ldots,\alpha_s\mid \alpha_1<\alpha_2<\ldots<\alpha_s\}
\subseteq \{1,2,\ldots,n\}$.
The element $(R(1))^k$ is a~sum of summands of the form
$S = \mu \sum\limits_{\sigma\in S_k}e_{\alpha_{\sigma(1)}}\wedge\ldots \wedge e_{\alpha_{\sigma(k)}}$.
By anticommutativity of the generators of the Grassmann algebra,
the necessary condition for~$S$ not to be zero is the following:
all numbers $|\alpha_1|,|\alpha_2|,\ldots,|\alpha_k|$ (maybe except only one)
are even. Thus, $(R(1))^k\in \wedge^{2k-1}(V)$.
Hence, we have $(R(1))^{[(n+1)/2]+1} \in \wedge^{n+1}(V) = (0)$.

(d) Follows from (c) and Lemma~\ref{lem11}(a).
\end{proof}

\begin{exm}[\cite{BGP}]
Let $F$ be a field of characteristic not two.
Up to conjugation with an automorphism,
all RB-operators of nonzero weight on $\mathrm{Gr}_{2}$
over $F$ with linear basis $1,e_1,e_2,e_1\wedge e_2$ are the following:
$R(1),R(e_1)\in\Span\{e_2,e_1\wedge e_2\}$,
$R(e_2) = R(e_1\wedge e_2) = 0$.
\end{exm}

\begin{cor}
Given an RB-operator $R$ of weight zero on $\mathrm{Gr}_2$
over a field $F$ with $\char F\neq2$, we have $R^2 = 0$ and
$R$~corresponds to Lemma~\ref{lem9}(a2). Moreover, $\rb(\mathrm{Gr}_2) = 2$.
\end{cor}

\begin{sta}\label{sta8}
Let $\char F\neq 2,3$. Given an RB-operator $R$ of nonzero
weight on $\mathrm{Gr}_{3}$, we have $R^3 = 0$.
\end{sta}

\begin{proof}
Applying Lemma~\ref{lem14}(b) and~(\ref{RB}), let us compute the following
for $x = 1+y$, $y\in A_0(3)$,
\begin{multline*}
0 = R(x)R(x)R(x)
  = R( R(x)R(x)x + R(x)xR(x) + xR(x)R(x) ) \\
  = 3R(R(x)R(x))
  = 3R^2(R(x)x+xR(x)) \\
  = 6R^3(x) + 3R^2(R(x)y+yR(x)) = 6R^3(x).
\end{multline*}
Using $R^3(1) = 1/6(R(1))^3 = 0$, we have done.
\end{proof}

\begin{exm}
Let a linear map $R$ on $\mathrm{Gr}_3$ be such that
$R(e_1\wedge e_2) = e_3$ and $R$ equals zero on all other basic elements.
The operator $R$ is an RB-operator of weight zero on $\mathrm{Gr}_3$
corresponds to Lemma~\ref{lem9}(a1) for $A_0 = \Span\{1,e_1\wedge e_2\}$
and $A_1$ equal to a linear span of all other basic elements.
\end{exm}

\begin{exm}\label{exmp1}
Let $\char F\neq 2,3$.
Define a linear map $R$ on $\mathrm{Gr}_3$ as follows:
$R(1) = e_1+e_2\wedge e_3$, $R(e_1) = e_1\wedge e_2\wedge e_3$
and $R$ equals zero on all other basic elements.
Then $R$ is an RB-operator of weight zero,
$R^3 = 0$ but $R^2\neq 0$. The operator~$R$ is constructed by Lemma~\ref{lem9}(h)
for $B = \Span\{1,e_1+e_2\wedge e_3,e_1\wedge e_2\wedge e_3\}$
and $C = \Span\{e_2,e_3,e_1\wedge e_2,e_1\wedge e_3,e_2\wedge e_3\}$,
where $R$ on $B$ corresponds to Lemma~\ref{lem9}(f).
\end{exm}

\begin{cor}\label{cor12}
Over a field $F$ with $\char F\neq 2,3$, we have $\rb(\mathrm{Gr}_3) = 3$.
\end{cor}

\begin{proof}
It follows from Statement~\ref{sta8} and Example~\ref{exmp1}.
\end{proof}

\begin{prob}
What is the precise value of $\rb(\mathrm{Gr}_n)$ for $n>3$?
\end{prob}

\subsection{Matrix algebra}

\begin{lem}\label{lem15}
Let $R$ be an RB-operator of weight zero on $M_n(F)$, $\char F = 0$, then

(a) \cite{BGP} $\Imm R$ consists only of degenerate matrices
and $\dim(\Imm R)\leq n^2-n$,

(b) $R^{2n} = 0$,

(c) $\dim(\Imm R)\leq (2n-1)\dim(\Imm R\cap \ker R)$.
\end{lem}

\begin{proof}
(b) By Lemma~\ref{lem10}, $R(1)$ is a nilpotent matrix. Thus, $(R(1))^n = 0$
and by Lemma~\ref{lem11}(a), $R^{2n} = 0$.

(c) It follows from (b).
\end{proof}

\begin{thm}[\cite{Aguiar00-2,BGP,Mat2}]\label{thm15}
All nonzero RB-operators of weight zero on $M_2(F)$ over an algebraically closed field $F$
up to conjugation with automorphisms $M_2(F)$, transpose
and multiplication on a nonzero scalar are the following:

(M1) $R(e_{21}) = e_{12}$, $R(e_{11}) = R(e_{12}) = R(e_{22}) = 0$;

(M2) $R(e_{21}) = e_{11}$, $R(e_{11}) = R(e_{12}) = R(e_{22}) = 0$;

(M3) $R(e_{21}) = e_{11}$, $R(e_{22}) = e_{12}$, $R(e_{11}) = R(e_{12}) = 0$;

(M4) $R(e_{21}) =- e_{11}$, $R(e_{11}) = e_{12}$, $R(e_{12}) = R(e_{22}) = 0$.
\end{thm}

\begin{cor}
All nonzero RB-operators of weight zero on $M_2(F)$ over an algebraically closed field~$F$
correspond to Lemma~\ref{lem9}.
\end{cor}

\begin{proof}
(M1) corresponds to Lemma~\ref{lem9}(a);
(M2), (M3) correspond to Lemma~\ref{lem9}(b);
(M4) corresponds to Lemma~\ref{lem9}(g) for
$B = \Span\{e_{21}\}$, $C = \Span\{e_{11}\}$,
$D = \Span\{e_{12},e_{22}\}$.
\end{proof}

\begin{rem}
If $R$ is an RB-operator on $M_2(F)$, then
$R$ is an RB-operator on $\textrm{gl}_2(\mathbb{C})$ by Statement~\ref{sta5}(a).
Since $\textrm{gl}_2(\mathbb{C}) = \textrm{sl}_2(\mathbb{C})\oplus \Span\{E\}$,
the projection of $R$ on $\textrm{sl}_2(\mathbb{C})$ is an RB-operator
by Lemma~\ref{lem2}. Thus, the projection of (M1) gives (L5),
the projection of (M2) gives (L4), and the projections of (M3), (M4)
coincide with (L1), where $t = 0$.
\end{rem}

\begin{exm}[\cite{GolubchikSokolov}]
Let $A$ be a subalgebra of $M_n(F)$ of matrices
with zero $n$-th row. Denote by $A_0$ the subalgebra
of $A$ consisting of matrices with zero $n$-th column
and by $A_1$ the~subspace of $A$ of matrices with all
zero columns except $n$-th. A linear map $R$
acting as follows: $R\colon A_0\to A_1$, $R\colon A_1 \to (0)$
is an RB-operator on $A$ by Lemma~\ref{lem9}(a1).
\end{exm}

\begin{exm}\label{exm23}
Let $R$ be an RB-operator of weight zero on $M_{n-1}(F)$,
then a linear operator $P$ of weight zero on $M_n(F)$ defined as follows:
$P(e_{ij}) = R(e_{ij})$, $1\leq i,j\leq n-1$,
$P(e_{in}) = P(e_{ni}) = 0$, $1\leq i\leq n$,
is an RB-operator of weight zero on $M_n(F)$ by Lemma~\ref{lem9}(h).
\end{exm}

In Corollary~\ref{cor1} all skew-symmetric RB-operators of weight zero on
$M_3(\mathbb{C})$ up to conjugation, transpose and scalar multiple were listed.

\begin{sta}\label{sta9}
All skew-symmetric RB-operators of weight zero on $M_3(\mathbb{C})$
except (R2) correspond to Lemma~\ref{lem9}.
\end{sta}

\begin{proof}
The RB-operator (R1) corresponds to Lemma~\ref{lem9}(a) for the space
$B =\Span\{e_{31},e_{32}\}$
and the linear span $C$ of all other matrix unities.
The RB-operator (R3) corresponds to Lemma~\ref{lem9}(h),
it is the extension of (M4) from Theorem~\ref{thm15} by Example~\ref{exm23}.

Finally, the RB-operators (R4)--(R8) correspond to Lemma~\ref{lem9}(g)
for $B = \Span\{e_{13},e_{23}\}$, $C =\Span\{e_{11},e_{12},e_{21},e_{22},e_{33}\}$
and $D = \Span\{e_{31},e_{32}\}$.
\end{proof}

Modifying and generalizing the RB-operator (R2), we get the following examples.

\begin{exm}\label{exm24}
For $A = M_n(F)$, define $r\in A\otimes A$ as
\begin{multline}\label{AYBE^Mat_n}
r = \sum\limits_{1\leq i\leq j\leq n-1}
(e_{ji}\otimes (e_{i,j+1}+e_{i+1,j+2}+\ldots+e_{n-j+i-1,n}) \\
-(e_{i,j+1}+e_{i+1,j+2}+\ldots+e_{n-j+i-1,n})\otimes e_{ji}).
\end{multline}
By the definition, $r$~is skew-symmetric. It is easy to check that~$r$
is a~solution of the associative Yang---Baxter equation~(\ref{AYBE}).
Up to sign, $r$~coincides with Example 2.3.3 from~\cite{Aguiar01}.
\end{exm}

\begin{exm}\label{exm25}
An RB-operator on $M_n(F)$ obtained from Example~\ref{exm24}
by Statement~\ref{sta4} is the following one:
\begin{equation}\label{MaxRbMat}
R(e_{ij}) = \begin{cases}
e_{i,j+1}+e_{i+1,j+2}+\ldots+e_{n-j+i-1,n}, & i\leq j\leq n-1, \\
-(e_{i-1,j}+e_{i-2,j-1}+\ldots+e_{i-j+1,1}), & i>j, \\
0, & j = n. \end{cases}
\end{equation}
Due to the definition, we have $R^{2n-1} = 0$ and $R^{2n-2}\neq0$, as
\begin{multline*}
(-1)^{n+1}e_{n,1}\to_R (-1)^n e_{n-1,1}\to_R (-1)^{n-1}e_{n-2,1}\to_R\ldots \\
\to_R e_{11}\to_R e_{12}+e_{23}+\ldots+e_{n-1,n}\to_R e_{13}+2e_{24}+3e_{35}+\ldots+(n-2)e_{n-2,n} \\
\ldots \to_R e_{1,n-1} + (n-2)e_{2,n}\to_R e_{1,n}\to_R 0.
\end{multline*}
\end{exm}

\begin{rem}\label{remGrad}
Note that the RB-operator from Example~\ref{exm25} preserves the well-known
\linebreak
$\mathbb{Z}$-grading  on $M_n(F)$:
$M_n(F) = \sum\limits_{i=-n+1}^{n-1}M_i$,
$M_i = \Span\{e_{st}\mid t-s = i\}$.
\end{rem}

\begin{rem}\label{rem9}
The linear operator defined by~(\ref{MaxRbMat}) on the ring of matrices $M_n(A)$
over an algebra~$A$ is an RB-operator on $M_n(A)$.
So we have inequality $2n-1\leq\rb(M_n(A))$ for every nonzero algebra~$A$.
\end{rem}

\begin{cor}
We have $2n-1\leq \rb(M_n(\mathbb{H}))\leq 2n$,
where $\mathbb{H}$ is the algebra of quaternions.
\end{cor}

\begin{proof}
We consider $M\in M_n(\mathbb{H})$ as an algebra over $\mathbb{R}$.
The lower bound follows from Remark~\ref{rem9}.
Given a nilpotent matrix $M\in M_n(\mathbb{H})$, we have $M^n = 0$ \cite{Wiegmann,Quater}.
Thus, the upper bound follows from Lemma~\ref{lem11}(a).
\end{proof}

\begin{thm}\label{thmRBMn}
Given a field $F$ of characteristic zero,
we have $\rb(M_n(F)) = 2n - 1$.
\end{thm}

\begin{proof}
From Lemma~\ref{lem15}(b) and Example~\ref{exm25}, we have
$2n-1\leq \rb(M_n(F))\leq 2n$. Let us extend the field $F$
to its algebraic closure $\bar{F}$
and prove the statement over $\bar{F}$, as the RB-index
does not decrease under a field extension.

Consider the Jordan normal form $J$ of the matrix $R(1)$.
Since $R(1)$ is nilpotent, we have either $J^{n-1} = 0$
or $J = e_{12} + e_{23} + \ldots + e_{n-1n}$.
In the first case, we apply Lemma~\ref{lem11}
and get $R^{2n-2} = 0$ as required.
It remains to study the second one.
Up to the conjugation with an automorphism of $M_n(F)$
we may assume that $R(1) = J$.
By~Lemma~\ref{lem1}(d),
\begin{equation}\label{JordanDegrees}
R(J^k) = \frac{J^{k+1}}{k+1}
\end{equation}
and so $R(e_{nn}) = 0$.

Consider the $\mathbb{Z}$-grading on $M_n(F)$ from Remark~\ref{remGrad}.
We deduce Theorem from the following result.

\begin{lem}
We have
$R\colon M_i\to \sum\limits_{j=i+1}^{n-1}M_j$
for all $i=-n+1,\ldots,n-2$.
\end{lem}

\noindent{\it Proof}.\ %\begin{proof}
We already know that $R$ maps $M_{n-1}$ to zero.
For $i = -n+1,\ldots,n-2$, we prove the statement by induction on~$i$
(decreasing from $n-2$ to~$-n+1$).

Let $i = n-2$. From
$$
R(1)R(e_{1n-1}) = R^2(e_{1n-1})
 = R(e_{1n-1})R(1)
$$
for $A = (a_{ij})_{i,j=1}^n = R(e_{1n-1})$, we conclude that
$$
\begin{pmatrix}
a_{21} & a_{22} & a_{23} & \ldots & a_{2n} \\
a_{31} & a_{32} & a_{33} & \ldots & a_{3n} \\
\hdotsfor{5} \\
a_{n1} & a_{n2} & a_{n3} & \ldots & a_{nn} \\
0 & 0 & 0 & 0 & 0 \\
\end{pmatrix} =
\begin{pmatrix}
0 & a_{11} & a_{12} & \ldots & a_{1n-1} \\
0 & a_{21} & a_{22} & \ldots & a_{2n-1} \\
\hdotsfor{5} \\
0 & a_{n-11} & a_{n-12} & \ldots & a_{n-1n-1} \\
0 & a_{n1} & a_{n2} & \ldots & a_{nn-1} \\
\end{pmatrix}.
$$
Thus,
$R(e_{1n-1}) = a_{11}E + a_{12}J + \ldots + a_{1n-1}J^{n-2} + a_{1n}J^{n-1}$.
Since $\Imm(R)$ contains only degenerate matrices, $a_{11} = 0$.
Further,
$$
R(e_{1n-1})R^2(1) = \frac{1}{2}R(e_{1n-1})J^2
 = \frac{1}{2}(a_{12}J^3 + a_{13}J^4 + \ldots + a_{1n-2}J^{n-1}),
$$
\begin{multline*}
R(e_{1n-1}R^2(1)) + R(R(e_{1n-1})R(1)) \\
 = R(a_{12}J^2 + a_{13}J^3 + \ldots + a_{1n-2}J^{n-2} + a_{1n-1}J^{n-1}) \\
 = \frac{a_{12}J^3}{3} + \frac{a_{13}J^4}{4} + \ldots + \frac{a_{1n-2}J^{n-1}}{n-1}.
\end{multline*}
Comparing coefficients, we have $R(e_{1n-1}) = a_{1n}J^{n-1} \in M_{n-1}$.
Since $R(e_{1n-1}$ $+e_{2n}) = e_{1n}/n$ by~\eqref{JordanDegrees},
we have $R(e_{2n})\in M_{n-1}$ as well.

Let us prove the inductive step for $i\geq0$.
Consider the matrix unity $e_{1 i+1}$,
the proof for other unities from $M_i$ is the same.
Denote $A = (a_{kl})_{k,l=1}^n = R(e_{1i+1})$. By
\begin{gather*}
R(1)R(e_{1i+1}) = R^2(e_{1i+1}), \\
R(e_{1i+1})R(1) = R^2(e_{1i+1}) + R(e_{1i+2})
\end{gather*}
we get
\begin{equation}\label{rb-index:RxR1}
\begin{pmatrix}
a_{21} & a_{22} & a_{23} & \ldots & a_{2n} \\
a_{31} & a_{32} & a_{33} & \ldots & a_{3n} \\
\hdotsfor{5} \\
a_{n1} & a_{n2} & a_{n3} & \ldots & a_{nn} \\
0 & 0 & 0 & 0 & 0 \\
\end{pmatrix}  \\
= \begin{pmatrix}
0 & a_{11} & a_{12} & \ldots & a_{1n-1} \\
0 & a_{21} & a_{22} & \ldots & a_{2n-1} \\
\hdotsfor{5} \\
0 & a_{n-11} & a_{n-12} & \ldots & a_{n-1n-1} \\
0 & a_{n1} & a_{n2} & \ldots & a_{nn-1} \\
\end{pmatrix} + B
\end{equation}
for $B\in \sum\limits_{j=i+2}^{n-1}M_j$. Thus,
$R(e_{1i+1}) = a_{11}E + a_{12}J + a_{13}J^2 + \ldots + a_{1i+1}J^i + B'$
for $B'\in \sum\limits_{j=i+1}^{n-1}M_j$.
Since $\Imm(R)$ contains only degenerate matrices, $a_{11} = 0$.
Further,
\begin{equation}\label{rb-index:indStepL}
R(e_{1i+1})R^2(1) = \frac{1}{2}R(e_{1i+1})J^2
 = \frac{1}{2}(a_{12}J^3 + \ldots + a_{1i+1}J^{i+2} + B'J^2).
\end{equation}
By the inductive hypothesis,
\begin{multline}\label{rb-index:indStepR}
R(e_{1i+1}R^2(1)) + R(R(e_{1i+1})R(1))
 =  B'' + R(a_{12}J^2 + \ldots + a_{1i+1}J^{i+1} + B'J) \\
 = \frac{a_{12}J^3}{3} + \frac{a_{13}J^4}{4} + \ldots + \frac{a_{1i+1}J^{i+2}}{i+2} + B'''
\end{multline}
for $B'' = R(e_{1i+1}R^2(1))\in \sum\limits_{j=i+3}^{n-1}M_j$ and
$B''' = R(B'J) + B''\in \sum\limits_{j=i+3}^{n-1}M_j$.
Comparing coefficients in~\eqref{rb-index:indStepL} and~\eqref{rb-index:indStepR},
we have $a_{12} = a_{13} = \ldots = a_{1i+1} = 0$
and $R(e_{1i+1}) \in \sum\limits_{j=i+1}^{n-1}M_j$.

Finally, let us prove the inductive step for $-n+1\leq i<0$.
Consider the matrix unity $e_{|i|+11}$,
the proof for other unities from $M_i$ is the same.
Denote $A = (a_{kl})_{k,l=1}^n = R(e_{|i|+11})$. From the formulas
\begin{gather*}
R(1)R(e_{|i|+11}) = R^2(e_{|i|+11}) + R(e_{|i|1}), \\
R(e_{|i|+11})R(1) = R^2(e_{|i|+11}) + R(e_{|i|+12}),
\end{gather*}
we again deduce~\eqref{rb-index:RxR1} by the inductive hypothesis.
It is easy to see that all coefficients $a_{kl}$
corresponding to the matrix unities
$e_{|i|+k\, j}$, $k = 1,\ldots,n-|i|$, $j\leq k$, are zero.
Thus, $R(e_{|i|+11})\in \sum\limits_{j=i+1}^{n-1}M_j$.
We have proved Lemma and the whole Theorem as well.
\end{proof}

For an algebra $A$ equal $A_1\oplus A_2\oplus \ldots\oplus A_k$,
we have $\rb(A) \geq \max\limits_{i=1,\ldots,k}\{\rb(A_i)\}$ by Lemma~\ref{lem9}(i).
Suppose that $A$ is a semisimple finite-dimensional associative algebra
over a field $F$ of characteristic zero and
$A = \bigoplus\limits_{i=1}^k A_i$ for $A_i \cong M_{n_i}(F)$.
By Lemma~\ref{lem11}(a) and Theorem~\ref{thmRBMn},
$\rb(A) \leq 2\max\limits_{i=1,\ldots,k}n_i
 = 1 + \max\limits_{i=1,\ldots,k}\{\rb(A_i)\}$.

\begin{con}\label{Conjecture}
Let $A$ be a semisimple finite-dimensional associative algebra
over a field $F$ of characteristic zero.
If $A = A_1\oplus \ldots \oplus A_k$ for simple $A_i$, then
$\rb(A) = \max\limits_{i=1,\ldots,k}\{\rb(A_i)\}$.
\end{con}

It is not difficult to verify that Conjecture~\ref{Conjecture}
holds for the case $A = M_2(F)\oplus M_2(F)$
and for the case $A = F\oplus F\oplus \ldots \oplus F$ (it is Corollary~\ref{cor8}).

\subsection{Simple Jordan algebras and composition algebras}

\begin{thm}[\cite{Gub2017}]\label{thm16}
(a) Let $J$ be a (not necessary simple or finite-dimensi\-onal)
Jordan algebra of a bilinear form over a field of characteristic not
two. Then $\rb(J)\leq3$.

(b) Let $F$ be an algebraically closed field of characteristic not two
and $J_{n+1}(f)$ be the simple Jordan algebra of a bilinear
form $f$ over $F$. We have
$\rb(J_{n+1}(f)) = \begin{cases}
2, & n = 2, \\
3, & n\geq3.
\end{cases}$
\end{thm}

An algebra $A$ over a field $F$ with $\char F\neq2$ is called
a~composition algebra \cite{Nearly} if there exists
a~nondegenerate quadratic form~$n$ on $A$ satisfying
$n(xy) = n(x)n(y)$, $x,y\in A$.

Every unital composition algebra is alternative and quadratic
and has a dimension 1, 2, 4 or 8 over $F$.

Moreover, a composition algebra $A$ is either a division algebra
or a split algebra, depending on the existence
of a nonzero $x\in A$ such that $n(x) = 0$.
A split composition algebra is either $F$, or $F\oplus F$, or $M_2(F)$, or $C(F)$,
the matrix Cayley---Dickson algebra. Let us give the definition
of the product on $C(F) = M_2(F)\oplus v M_2(F)$.
We extend the product from
$M_2(F)$ on $C(F)$ as follows:
\begin{equation}\label{CD-product}
a(vb) = v(\bar{a}b),\quad (vb)a = v(ab),\quad
(va)(vb) = b\bar{a},\quad a,b\in M_2(F).
\end{equation}
Here $\bar{a}$ for $a = (a_{ij})\in M_2(F)$ means the matrix
$\begin{pmatrix}
a_{22} & - a_{12} \\
-a_{21} & a_{11} \end{pmatrix}$.

\begin{thm}
Given a unital composition algebra $C$ over a~field~$F$
of characteristic zero, we have
\begin{equation}
\rb(C) = \begin{cases}
1, & C\ \mbox{is a division algebra or }\dim(C)\leq2, \\
3, & C\ \mbox{is split and}\ \dim(C)= 4,8.
\end{cases}
\end{equation}
\end{thm}

\begin{proof}
If $C$ is a division composition algebra, then we have
$\rb(C) = 1$ by Theorem~\ref{thm1}.
If $C$ is a split composition algebra with $\dim C\leq2$, then
$C = F$ or $C \cong F\oplus F$. By Corollary~\ref{cor8}, $\rb(C) = 1$.

If $C$ is a split composition algebra with $\dim C = 4$, then
$C \cong M_2(F)$ and we are done by Statement~\ref{thmRBMn}.

Finally, let $C$ be a split composition algebra with
$\dim C = 8$, i.e., $C = C(F)$, the matrix Cayley---Dickson algebra.
By Statement~\ref{sta5}(a), each RB-operator on $C(F)$ is an RB-operator
on $C(F)^{(+)}$ which is isomorphic to the
simple Jordan algebra of a bilinear form~\cite[p. 57]{Nearly}.
By Theorem~\ref{thm16}(a), $\rb(C(F))\leq3$.
Extending the RB-operator (M4) from $M_2(F)$ on $C(F)$ by Lemma~\ref{lem9}(h),
we have $\rb(C(F)) = 3$.
\end{proof}

Every simple finite-dimensional Jordan algebra is either
simple division Jordan algebra or of Hermitian, Albert or
Clifford type \cite{McCrimmon}.
If $J$~is a simple finite-dimensional division Jordan algebra,
then $\rb(J) = 1$ by Theorem~\ref{thm2}. The case of simple Jordan algebras
of Clifford type was considered in \cite{Gub2017}, see Theorem~\ref{thm16}.
For one of the cases of Hermitian type, we can say the following

\begin{sta}
We have

(a) $2n-1 \leq \rb( M_n(F)^{(+)})\leq 3n-1$,where $\char F = 0$,

(b) $2n-1 \leq \rb( M_n(\mathbb{H})^{(+)})\leq 3n-1$.
\end{sta}

\begin{proof}
(a) The lower bound follows from Statement~\ref{sta5}(a) and Example~\ref{exm25},
the upper bound follows from Lemma~\ref{lem11}(b).

(b) By the theory of matrices over~$\mathbb{H}$~\cite{Wiegmann,Quater},
the proof is analogous.
\end{proof}

Finally, we study simple Jordan algebras of Albert type.
Over an algebraically closed field $F$, the only simple Jordan algebra
of Albert type is $H_3(C(F))$, the space of Hermitian matrices
over $C(F)$ under the product $a\circ b = ab + ba$.

\begin{thm}
Let $F$ be an algebraically closed field of characteristic zero,
$A$ be the simple Jordan algebra over $F$ of Albert type, then
$5\leq \rb(A)\leq 8$.
\end{thm}

\begin{proof}
As every Jordan algebra of Albert type is cubic, i.e., satisfies a~cubic equation,
so $(R(1))^3 = 0$ for any RB-operator $R$ of weight zero on~$A$.
By Lemma~\ref{lem11}(b), $\rb(A)\leq 8$.
Due to the first Tits construction~\cite[Chap. IX,\,\S~12]{Jacobson},
we have $A = A_1 \oplus A_2$ (as vector spaces) for
$A_1\cong M_3(F)^{(+)}$ and $A_1$-module $A_2$,
Applying Lemma~\ref{lem9}(h) and Statement~\ref{sta5}(a), we have the lower bound
$\rb(A)\geq \rb(M_3(F)) = 5$.
\end{proof}

\begin{rem}
By Example~\ref{exm18}, we have $\rb(\mathrm{K}_3) = 2$
for the Kaplansky superalgebra~$\mathrm{K}_3$.
\end{rem}

\section*{Acknowledgements}

Author is grateful to P.S. Kolesnikov for
the helpful discussions, comments, corrections, and the joint Theorem~\ref{thm4}.
Also, author thanks V.N. Zhelyabin, Yu.N. Maltsev, and Yi Zhang for useful comments.

The author was supported by the Austrian Science Foundation FWF grant P28079 and
by the Program of fundamental scientific researches of
the Siberian Branch of Russian Academy of Sciences, I.1.1, project 0314-2016-0001.

\noindent Vsevolod Gubarev \\
University of Vienna \\
Oskar-Morgenstern-Platz 1, 1090 Vienna, Austria \\
Sobolev Institute of Mathematics \\
Acad. Koptyug ave. 4, 630090 Novosibirsk, Russia \\
e-mail: wsewolod89@gmail.com
\end{document}